\documentclass[reqno]{amsart}
\usepackage{amsmath,amsthm,amsfonts, amssymb,amscd,enumitem}
\usepackage[all]{xy}
\numberwithin{equation}{section}
  \newtheorem{thm}{Theorem}[section]
  \newtheorem{lem}[thm]{Lemma}
  \newtheorem{prop}[thm]{Proposition}
  \newtheorem{cor}[thm]{Corollary}
  \theoremstyle{definition}
  \newtheorem{defn}[thm]{Definition}
  \newtheorem{exm}[thm]{Example}
  \newtheorem{rmk}[thm]{Remark}

\newcommand{\lex}{\,\overrightarrow{\times}\,}

\newcommand{\Rad}{\mbox{\rm Rad}}

 \newcommand\ra{\rightarrow}

 \newcommand\mI{\mathcal{I}}

 \newcommand\thi{\theta\!_{_I}}
 \newcommand\s{\subseteq}
 
 \newcommand\supp{\mathrm{Supp}}

 \newcommand\B{\mathrm{Br}}
 
  \newcommand\lam{\lambda}

\def\iff{if and only if }
\let\Right\right
\let\Left\left
\makeatletter
\def\right#1{\Right#1\@ifnextchar){\!\right}{}}
\def\left#1{\Left#1\@ifnextchar({\!\left}{}}

\begin{document}
\title[A Variety Containing EMV-Algebras and Pierce Sheaves]{A Variety Containing EMV-Algebras and Pierce Sheaves of EMV-algebras}
\author[Anatolij Dvure\v{c}enskij, Omid Zahiri]{Anatolij Dvure\v{c}enskij$^{^{1,2}}$, Omid Zahiri$^{^{3}}$}

\date{}%
\thanks{}
\address{$^1$Mathematical Institute, Slovak Academy of Sciences, \v{S}tef\'anikova 49, SK-814 73 Bratislava, Slovakia}
\address{$^2$Palack\' y University Olomouc, Faculty of Sciences, t\v r. 17.listopadu 12, CZ-771 46 Olomouc, Czech Republic}
\address{$^3$ Tehran, Iran}
\email{dvurecen@mat.savba.sk, zahiri@protonmail.com}
\thanks{A.D. acknowledges the support of the grant of
the Slovak Research and Development Agency under contract APVV-16-0073 and the grant VEGA No. 2/0069/16 SAV, and O.Z. is thankful Mathematical Institute SAS for its hospitality during my stay at the Institute}

\keywords{Variety, EMV-algebra, wEMV-algebra, associated wEMV-algebra, Pierce sheaf, representation of wEMV-algebras}
\subjclass[2010]{06C15, 06D35}


\begin{abstract}
According to \cite{Dvz}, we know that the class of all EMV-algebras, $\mathsf{EMV}$, is not a variety, since it is not closed under the subalgebra operator. The main aim of this work is to find the least variety containing $\mathsf{EMV}$. For this reason, we introduced the variety $\mathsf{wEMV}$ of wEMV-algebras of type $(2,2,2,2,0)$ induced by some identities. We show that, adding a derived binary operation $\ominus$ to each EMV-algebra $(M;\vee,\wedge,\oplus,0)$, we extend its language, so that $(M;\vee,\wedge,\oplus,\ominus,0)$, called an associated wEMV-algebra, belongs to $\mathsf{wEMV}$. Then using the congruence relations induced by the prime ideals of a wEMV-algebra, we prove that each wEMV-algebra can be embedded into an associated wEMV-algebra.
We show that $\mathsf{wEMV}$ is the least subvariety of the variety of wEMV-algebras containing $\mathsf{EMV}$. Finally, we study Pierce sheaves of proper EMV-algebras.
\end{abstract}

\date{}

\maketitle

\section{ Introduction}

C.C. Chang \cite{chang1,chang2} introduced MV-algebras to provide an algebraic proof of the completeness theorem of \L ukasiewicz's infinite-valued propositional calculus. D. Mundici \cite{Mun} proved that there is a categorical equivalence between the category of unital Abelian $\ell$-groups and the category of MV-algebras. Today, the theory of MV-algebras is very deep and has many connections with
other algebraic structures and other parts of mathematics with many important applications to different areas (for more details see \cite{cdm, mundici 2}). It is well known that MV-algebras form a variety.
In \cite{Kom}, Y. Komori has described all subvarieties of the variety of MV-algebras. He proved that the lattice of subvarieties of the variety \textsf{MV} of MV-algebras is countably infinite. A. Di Nola and A. Lettieri presented in \cite{DiLe1} an equational base of any subvariety of the variety \textsf{MV} which consists of finitely many MV-equations.

Recently in \cite{Dvz}, we introduced EMV-algebras to generalize MV-algebras and generalized Boolean algebras. An EMV-algebra locally resembles MV-algebras, but a top element is not guaranteed.
Conjunction and disjunction exist but negation exists only in a local sense, i.e. negation of $a$ in $b$ exists whenever $a\le b$, but the total negation of the event $a$ is not assumed. There is an interesting representation for EMV-algebras. Indeed, an EMV-algebra  either has a top element or we can find an EMV-algebra $N$ with top element where the original EMV-algebra can be embedded as a maximal ideal of the EMV-algebra $N$, \cite[Thm 5.21]{Dvz}. This result is crucial for our reasoning. The Loomis--Sikorski theorem for these algebras was established in \cite{Dvz1}. States as analogues of finitely additive measures were investigated in \cite{Dvz2} and morphisms and free EMV-algebras were described in \cite{Dvz3}. In \cite{Dvz4}, we showed that every EMV-algebra $M$ is a homomorphic image of $\B(M)$, the generalized Boolean algebra $R$-generated by $M$, where a homomorphism is a homomorphism of generalized effect algebras. Unfortunately, as we proved in \cite{Dvz},
the class of all EMV-algebras, $\mathsf{EMV}$, does not form a variety, since it is not closed under the subalgebra operator.

The main aim of the paper is two-fold: (1) investigate $\mathsf{EMV}$ in order to find the least variety ``containing" $\mathsf{EMV}$, (2) and study Pierce sheaves of EMV-algebras.

(1) For the first purpose, we introduced a new class of algebras, wEMV-algebras, of type (2,2,2,2,0) which form a variety. The language of every EMV-algebra can be naturally extended by a derived binary operation $\ominus$, so we obtain an associated wEMV-algebra corresponding to the original EMV-algebra.
The class of $\mathsf{EMV}_a$ of such associated wEMV-algebras is a proper subclass of wEMV-algebras. Then we define a strict wEMV-algebra and use it to show that each wEMV-algebra $M$ can be embedded into an associated wEMV-algebra. Also, $M$ can be embedded into a direct product of an associated EMV-algebra and a strict wEMV-algebra. Then it is proved that the class of wEMV-algebras is the least subvariety of the variety $\mathsf{wEMV}$ containing $\mathsf{EMV}_a$. We show that every wEMV-algebra without top element can be embedded into an associated wEMV-algebra with top element as its maximal ideal. We describe all subvarieties of wEMV-algebras and we show that they are only countably many.

(2) For the second goal, we note that if $M$ is a bounded EMV-algebra, then $M$ is termwise equivalent to an MV-algebra on $M$, and the Pierce representation of MV-algebras is studied in \cite[Part 4]{georgescu}.
So, in the article, we concentrate to a case when an EMV-algebra $M$ does not have a top element (proper EMV-algebra). Thus we construct a Pierce sheaf of a proper EMV-algebra and  we show that a global section of $T=(E_M,\pi,X)$ forms an EMV-algebra. It is proved that a semisimple EMV-algebra under a suitable condition can be embedded into a sheaf of bounded EMV-algebras on the space $X$. Finally, we show that if $(M;\vee,\wedge,\oplus,0)$ is a Stone EMV-algebra, then it can be embedded into the MV-algebra of global sections of a Hausdorff Boolean sheaf whose stalks are MV-chains.

The paper is organized as follows. In Section 2, we gather basic facts on EMV-algebras and their Basic Representation Theorem. Section 3 defines wEMV-algebras which form a variety whereas the class of EMV-algebras not. For every EMV-algebra, we extend its language adding a new derived binary operation $\ominus$, so that it is a wEMV-algebra associated to the original EMV-algebra, and we show that the variety of wEMV-algebras is the least subvariety of the variety of wEMV-algebras containing all associated wEMV-algebras. In addition, we study a decomposition of wEMV-algebras as a direct product. In the last section, we study the Pierce sheaves of EMV-algebras.

\section{Preliminaries}

Recently, we have introduced in \cite{Dvz} a common extension of generalized Boolean algebras and of MV-algebras called EMV-algebras.

An algebra $(M;\vee,\wedge,\oplus,0)$ of type $(2,2,2,0)$ is said to be an
{\it extended MV-algebra}, an {\it EMV-algebra} in short, if it satisfies the following conditions:
	
	\begin{itemize}
	\item[{\rm (E1)}] $(M;\vee,\wedge,0)$ is a distributive lattice with the least element $0$;
	\item[{\rm (E2)}]  $(M;\oplus,0)$ is a commutative ordered monoid (with respect to the lattice ordering (E1), i.e. $a\le b$ implies $a\oplus c \le b\oplus c$ for each $c\in M$) with a neutral element $0$;
	\item[{\rm (E3)}] for each $a\in \mI(M):=\{x\in M\mid x\oplus x=x\} $, the element
	$\lambda_{a}(x)=\min\{z\in[0,a]\mid x\oplus z=a\}$
	exists in $M$ for all $x\in [0,a]$, and the algebra $([0,a];\oplus,\lambda_{a},0,b)$ is an MV-algebra;
	\item[{\rm (E4)}] for each $x\in M$, there is $a\in \mI(M)$ such that
	$x\leq a$.
\end{itemize}
\vspace{1mm}
An EMV-algebra $M$ is called {\em proper} if it does not have a top element. Clearly, if in a generalized Boolean algebra we put $\oplus =\vee$, $\lambda_a(x)$ is a relative complement of $x \in [0,a]$ in $[0,a]$, every generalized Boolean algebra can be viewed as an EMV-algebra where top element is not necessarily assumed. If $1$ is a top element of an EMV-algebra $M$, by (E3), $([0,1];\oplus,\lambda_1,0,1)=(M;\oplus,',0,1)$ is an MV-algebra. Conversely, if $(M;\oplus,',0,1)$ is an MV-algebra, then $(M;\vee,\wedge,\oplus,0)$ is an EMV-algebra with top element $1$. In addition, every EMV-algebra $(M;\vee,\wedge,\oplus,0)$ with top element $1$ is termwise equivalent to an MV-algebra $(M;\oplus,',0,1)$.

We note that in \cite{Dvz}, there a list of interesting examples of EMV-algebras.

Let $(M;\vee,\wedge,\oplus,0)$ be an EMV-algebra. Its reduct
$(M;\vee,\wedge,0)$ is a distributive lattice with a bottom element $0$. The lattice structure of $M$  yields a partial order relation on $M$, denoted by $\leq$, that is $x\leq y$ iff $x\vee y=y$ iff $x\wedge y=x$. Also, if $a$ is a fixed idempotent element of $M$, there is a
partial order relation $\preccurlyeq_a$  on
the MV-algebra $([0,a];\oplus,\lambda_a,0,a)$ defined by $x\preccurlyeq_a y$ iff $\lambda_a(x)\oplus y=a$.
By \cite{Dvz3}, we know that for each $x,y\in [0,a]$, we have
\[x\leq y\Leftrightarrow x\preccurlyeq_a y.\]
In addition, if also $x,y \le b\in \mathcal I(M)$, then
\[x\preccurlyeq_a y \Leftrightarrow  x\leq y\Leftrightarrow x\preccurlyeq_b y.\]

\begin{prop}\label{2.2}{\rm \cite[Prop 3.9]{Dvz}}
	Let $(M;\vee,\wedge,\oplus,0)$ be an EMV-algebra and $a,b\in \mI(M)$ such that $a\leq b$.
	Then  for each $x\in [0,a]$, we have
	
	\begin{itemize}
		\item[{\rm (i)}]  $\lam_a(x)=\lam_b(x)\wedge a$;
		\item[{\rm (ii)}]   $\lam_b(x)=\lam_a(x)\oplus \lam_b(a)$;
		\item[{\rm (iii)}] $\lam_b(a)$ is an idempotent, and $\lam_a(a)=0$.
	\end{itemize}
\end{prop}

\begin{lem}\label{2.3}{\rm \cite[Lem 5.1]{Dvz}}
Let $(M;\vee, \wedge,\oplus,0)$ be an EMV-algebra. For all $x,y\in M$, we define
\begin{equation}\label{eq:odot1}
x\odot y=\lam_a(\lam_a(x)\oplus \lam_a(y)),
\end{equation}
where $a\in\mI(M)$ and $x,y\leq a$. Then $\odot:M\times M\ra M$ is an order preserving, associative well-defined binary operation on $M$ which does not depend on $a\in \mI(M)$ with $x,y \le a$.
In addition, if $x,y \in M$, $x\le y$, then $y \odot \lambda_a(x)=y\odot \lambda_b(x)$
for all idempotents $a,b$ of $M$ with $x,y\le a,b$.
\end{lem}

The following important result on representing EMV-algebras was established in \cite[Thm 5.21]{Dvz}, it generalizes an analogous result for generalized Boolean algebras, see \cite{CoDa}.

\begin{thm}\label{2.4}{\rm [Basic Representation Theorem]}
Every EMV-algebra $M$ either has a top element or $M$ can be embedded into an EMV-algebra $N$  with top element as a maximal ideal of $N$ such that every element $x\in N$ is either the image of some element from $M$ or $x$ is the complement of the image of some element from $M$.
\end{thm}

The EMV-algebra $N$ with top element in the latter theorem is unique up to isomorphism and it is said to be {\it representing} the EMV-algebra $M$.
For more details, we refer to \cite{Dvz}.

An {\it MV-algebra} is an algebra $(M;\oplus,',0,1)$ (henceforth written simply as $M=(M;\oplus,',0,1)$) of type $(2,1,0,0)$, where $(M;\oplus,0)$ is a
commutative monoid with the neutral element $0$ and for all $x,y\in M$, we have:

\begin{enumerate}
	\item[(i)] $x''=x$;
	\item[(ii)] $x\oplus 1=1$;
	\item[(iii)] $x\oplus (x\oplus y')'=y\oplus (y\oplus x')'$.
\end{enumerate}
\noindent
In any MV-algebra $(M;\oplus,',0,1)$,
we can also define the following  operations:
\begin{equation}\label{eq:oplus}
x\odot y:=(x'\oplus y')',\quad x\ominus y:=(x'\oplus y)'.
\end{equation}
\noindent
We note that any MV-algebra is a distributive lattice where $x\oplus (x\oplus y')'=x\vee y = y\oplus (y\oplus x')'$ and $x\wedge y =x\odot (x'\oplus y)$.
Prototypical examples of MV-algebras are connected with unital $\ell$-groups, i.e. with couples $(G,u)$, where $G$ is an Abelian $\ell$-group with a fixed strong unit of $u\in G^+$. If we set $[0,u]=\{g\in G\mid 0\le g\le u\}$, then $\Gamma(G,u)=([0,u];\oplus,',0,u)$, where $x\oplus y := (x+y)\wedge u$ and $a':=u-x$, $x,y \in [0,u]$, is an MV-algebra, and every MV-algebra is isomorphic to a unique $\Gamma(G,u)$.

\section{A variety containing EMV-algebras}

As we showed in \cite[Sec 3]{Dvz}, the class of EMV-algebras, $\mathsf{EMV}$, is not closed under subalgebras so it is neither a variety nor a quasivariety with respect to the original EMV-operations. In the section, we introduce a new class of algebras called wEMV-algebras. If we extend the language of EMV-algebras adding a derived binary operation $\ominus$, we obtain an associated wEMV-algebra, and the variety of wEMV-algebras contains the class of associated wEMV-algebras, $\mathsf{EMV}_a$. We show that
this class is the least subvariety of the variety $\mathsf{sEMV}$ containing $\mathsf{EMV}_a$. In addition, we find some properties of this class as a direct decomposition of an wEMV-algebra to two its factors.

\begin{defn}\label{WEMV}
An algebra $(M;\vee,\wedge,\oplus,\ominus,0)$ of type (2,2,2,2,0) is called a {\em wEMV-algebra} (w means weak) if it
satisfying the following conditions:
\begin{enumerate}
\item[(i)] $(M, \vee, \wedge, 0)$ is a distributive lattice with the least
element $0$;
\item[(ii)] $(M; \oplus, 0)$ is a commutative monoid;
\item[(iii)] $(x \oplus y) \ominus x \leq y$;
\item[(iv)] $x \oplus (y \ominus x) = x \vee y$;
\item[(v)] $x \ominus (x \wedge y) = x \ominus y$;
\item[(vi)] $z \ominus (z \ominus x) = x \wedge z$;
\item[(vii)] $z \ominus (x \vee y) = (z \ominus x) \wedge (z\ominus y)$;
\item[(viii)]  $(x\wedge y)\ominus z=(x\ominus z)\wedge (y\ominus z)$;
\item[(ix)] $x\ominus (y\oplus z)=(x\ominus y)\ominus z$;
\item[(x)] $x\oplus (y\vee z)=(x\oplus y)\vee (x\oplus z)$.
\end{enumerate}
\end{defn}

An {\it idempotent} of a wEMV-algebra $M$ is any element $x\in M$ such that $x=x\oplus x$. We denote by $\mathcal I(M)$ the set of all idempotents of $M$, then $0\in \mathcal I(M)$. It can happen that $\mathcal I(M)=\{0\}$ as in Example \ref{ex:1} below.

In the following, we present some important examples of wEMV-algebras.

\begin{exm}\label{ex:1}
If $M=G^+$ is the positive cone of an Abelian $\ell$-group $G$, and if we define on $G^+$ two operations $x\oplus y:= x +y$ and $x\ominus y :=(x - y)\vee 0$, $x,y \in G^+$, then $(M;\vee,\wedge,\oplus,\ominus,0)$ is an example of a wEMV-algebra, called also a {\it wEMV-algebra of a positive cone}. Moreover, it can be embedded into the MV-algebra $N:=\Gamma(\mathbb Z \lex G,(1,0))$ as an maximal ideal of $N$, and every element of $N$ is either $(0,g)$ for some $g \in G^+$ or $(0,g)'=(1,0)-(0,g)$ for some $g\in G^+$.
\end{exm}

\begin{exm}\label{ex:2}
Let $(M;\oplus,',0,1)$ be an MV-algebra. If we set $x\ominus y:= x\odot y'$, $x,y \in M$, then $(M;\vee,\wedge,\oplus,\ominus,0)$
is a wEMV-algebra with a top element $1$.
\end{exm}

\begin{exm}\label{ex:3}
Consider an arbitrary proper EMV-algebra $(M;\vee,\wedge,\oplus,0)$. By Theorem \ref{2.4},
$M$ can be embedded into an EMV-algebra $N_0$ with top element as a maximal ideal of $N_0$. Then
$(N_0;\oplus,\lambda_1,0,1)$ is an MV-algebra. For simplicity, we use $x'$ instead of $\lambda_1(x)$, for all $x\in N_0$.
Let $\ominus$ be the well-known operation on $N_0$, that is
$x\ominus y=(x'\oplus y)'$ for all $x,y\in N_0$. Since $M$ is an ideal of $N_0$, then $M$ is closed under $\ominus$.
So, we have an example of a wEMV-algebra without top element.
\end{exm}

Combining Examples \ref{ex:2}--\ref{ex:3} with the Basic Representation Theorem, we see that if $(M;\vee,\wedge,\oplus,0)$ is an arbitrary EMV-algebra, extending its language with a binary operation $\ominus$, we obtain a wEMV-algebra $(M;\vee,\wedge,\oplus,\ominus,0)$; it is said to be a {\it wEMV-algebra associated} with the EMV-algebra $(M;\vee,\wedge,\oplus,0)$; simply we say $M$ is an {\it associated wEMV-algebra}. We denote by $\mathsf{EMV}_a$ the class of associated wEMV-algebras $(M;\vee,\wedge,\oplus,\ominus,0)$, where $(M;\vee,\wedge,\oplus,0)$ is any EMV-algebra. By a way, it is possible to show that if $x,y \le a$, where $a$ is an idempotent of $M$, then
\begin{equation}\label{eq:ominus}
x\ominus y = \lambda_a(\lambda_a(x)\oplus y),
\end{equation}
and it does not depend on the idempotent $a$. The same is true if we define
\begin{equation}\label{eq:odot}
x\odot y =\lambda_a(\lambda_a(x)\oplus\lambda_a y).
\end{equation}

\begin{exm}\label{exm3.14}
Let $\{(M_i;\vee,\wedge,\oplus,\ominus,0)\}$ be a family of wEMV-algebras. Then we can easily prove that
$\sum_{i\in I}M_i=\{f\in \prod_{i\in I}M_i\mid \supp(f) \mbox{ is finite}  \}$
with the componentwise operations form a wEMV-algebra. Recall that $\supp(f)=\{i\in I\mid f(i)\neq 0\}$.
\end{exm}

Let $a$ be an element of a wEMV algebra $(M;\vee,\wedge,\oplus,\ominus,0)$. We define
$$
0.a=0,\quad (n+1).a= (n.a)\oplus a,\ n\ge 0.
$$
Now, we present some properties of wEMV-algebras.

\begin{prop}\label{pr:1}
Basic properties of a wEMV-algebra are as follows:

\begin{enumerate}
\item[{\rm (a)}] $(M;\oplus,0)$ is an ordered monoid, i.e. $x\le y$ implies $x\oplus z\le y\oplus z$ for each $z\in M$. Moreover, $x\le y$ iff there is $a\in M$ such that $y=x \oplus a$.
\item[{\rm (b)}] If $x\le y$, then $x\ominus z\le y\ominus z$ for each $z\in M$. If $z_1\le z_2$, then $x\ominus z_2\le x\ominus z_1$ for each $x \in M$.

\item[{\rm (c)}] $z\ominus x\le z$ for all $x,z\in M$.
\item[{\rm (d)}]  $x\vee y\le x\oplus y$, $x,y \in M$.
\item[{\rm (e)}] If $x\le z$, then $(z\ominus x)\oplus x = z$ and $(z\ominus (z\ominus x))=x$
\item[{\rm (f)}] If $x\le z$, then
$z\ominus x= \min\{y\in [0,z]\colon x\oplus y=z\}$.
\item[{\rm (g)}] For all $x,z \in M$, we have $z\ominus x=\min\{t \in [0,z]\colon t\oplus (z\wedge x)=z\}$.

\item[{\rm (h)}] For $x,y,z \in M$, there holds $z \le x\oplus y$ if and only if $z\ominus x\le y$.

\item[{\rm (i)}] $z\ominus 0=z$ and $z\ominus z =0$ for each $z\in M$. Moreover, $z\ominus y=0$ if and only if $z\le y$.

\item[{\rm (j)}] If $x\le z$ and $z\ominus x=0$, then $z=x$.

\item[{\rm (k)}] If $a$ and $b$ are two different atoms of $M$, then $a\vee b = a\oplus b$.
\item[{\rm (l)}] If $a_1,\ldots,a_n$ are mutually different atoms of $M$, then $a_1\vee \cdots \vee a_n = a_1\oplus \cdots \oplus a_n$.
\end{enumerate}
\end{prop}

\begin{proof}
(a) Definition \ref{WEMV}(x) implies that $(M;\oplus,0)$ is an ordered monoid. Let $x\le y$, then by (iv), $y=x\vee y= x\oplus (y\ominus x)$; we put $a=y\ominus x$. Conversely, let $y= x\oplus a$ for some $a \in M$. Then
$x\vee y = x\oplus (y\ominus x)\ge x\oplus 0=x$, i.e. $x\le x\oplus y$. Similarly, $y\le x\oplus y$.  Now, let $y=x\oplus a$ for some $a\in M$. Then $y=x\oplus a\ge x\oplus 0=x$.

(b) It follows from (vii) and (viii).

(c) We have $z\wedge (z\ominus x)=z\ominus (z\ominus (z\ominus x))=z\ominus (x\wedge z)=z\ominus x$, i.e. $z\ominus x\le z$.

(d) $x\oplus y\ge x\oplus 0=x$, i.e. $x\le x\oplus y$. Similarly, $y\le x\oplus y$. Hence, $x\vee y \le x\oplus y$.

(e) Applying (iv), we have $ (z\ominus x)\oplus x=z\vee x=z$.

(f) By (e), $z\ominus x \in \{y\in [0,z] \mid x \oplus y=z\}$. If $x\oplus y=z$, then  $z\leq x\oplus y$ which implies that
\[(z\ominus x)\wedge ((x\oplus y)\ominus x)=(z \wedge (x\oplus y))\ominus
x=z\ominus x.\]
That is, $z\ominus x\leq (x\oplus y)\ominus x\leq y$, by (iii). Hence, (f) is proved.

(g)  Let $x,y \in M$. By (v), $z\ominus x = z\ominus (x\wedge z)$. Applying (f), we have establish (g).

(h) Let $z\le x \oplus y$. By (b), $z\ominus x\le (x\oplus y)\ominus x \le y$, by (iii).

Conversely, let $z\ominus x \le y$. Using (iv), we get $z\le z\vee x = (z\ominus x)\oplus x \le y \oplus x \le x \oplus y$.

(i) Check $z\ominus 0= (z\ominus 0)\oplus 0 = z\vee 0 = z$. On the other side, $z\ominus z = (z\oplus 0)\ominus z \le 0$. The second part follows from (v) and the first part of the present proof of (i): $z\ominus y = z\ominus (z\wedge y)=0$ iff $z\wedge y = z$.

(j)  Let $x\le z$ and $z\ominus x=0$, then by (e) and (i), we have $x=z\ominus (z\ominus x)= z\ominus 0=x$.

(k) Due to (b), we have $(a\vee b) \ominus a \le (a\oplus b)\ominus a \le b$. If $(a\vee b) \ominus a=0$, then (j) entails $a=a\vee b$, so that $b\le a$ which means $a=b$, a contradiction. Whence, $(a\vee b) \ominus a=b$. Then $a\vee b= ((a\vee b) \ominus a)\oplus a = b\oplus a=a\oplus b$.

(l) We proceed by induction. Due to (k), the statement holds for $n=2$. Assume that it holds for each integer $i\le n$, i.e. $a_1\vee \cdots \vee a_i=a_1\oplus \cdots \oplus a_i$. Set $b_n= a_1\vee \cdots \vee a_n = a_1\oplus \cdots \oplus a_n$. Check
$$
(b_n \vee a_{n+1})\ominus b_n \le (b_n \oplus a_{n+1})\ominus b_n \le a_{n+1}.
$$
There are two cases: First $(b_n \vee a_{n+1})\ominus b_n= 0$. Then by (j), $b_n\vee a_{n+1} = b_n$ and $a_{n+1}\le b_n = a_1\vee\cdots\vee a_n$.   Distributivity implies $a_{n+1}= (a_{n+1}\wedge a_1) \vee \cdots\vee (a_{n+1}\wedge a_n) = 0$ which is a contradiction. Therefore, we have the second case $(b_n \vee a_{n+1})\ominus b_n= (b_n\oplus a_{n+1})\ominus b_n = a_{n+1}$ which yields
$$
b_n \vee a_{n+1} = ((b_n \vee a_{n+1})\ominus b_n)\oplus b_n = a_{n+1}\oplus b_n
$$
as claimed.
\end{proof}

\begin{lem}\label{le:5.2}
Let $a$ be an atom of a wEMV-algebra $(M;\vee,\wedge,\oplus,\ominus,0)$ and $b$ an arbitrary element of $M$.
If there is an integer $n\ge 0$ such that $b\le n.a$, then $b = m.a$ for some $m\ge 0$.
\end{lem}

\begin{proof}
(1) We first show that either $(n+1).a$ is a cover of $n.a$ for each $n\ge 0$ or $n.a$ is an idempotent for some $n\ge 1$.
Assume that there is an integer $n\ge 0$ such that $n.a< b\le (n+1).a$. We show that $b=(n+1).a$. By Proposition \ref{pr:1}(a), there is an element $c\in M$ such that $b= (n.a)\oplus c$. Hence, $((n.a)\oplus c)\ominus (n.a)\le (a\oplus n.a)\ominus (n.a)\le a$. There are two cases: Either $((n.a)\oplus c)\ominus (n.a) = 0$ or $((n.a)\oplus c)\ominus (n.a)=a$. In the first one, we have
\begin{align*}
n.a &= \left(\left(\left(n.a\right)\oplus c\right)\ominus \left(n.a\right)\right)\oplus (n.a)\\
 &=\left(\left(n.a\right)\oplus c\right)\vee \left(n.a\right)= \left(n.a\right)\oplus c = b,
\end{align*}
which is a contradiction. Hence, we have the second case $((n.a)\oplus c)\ominus (n.a)=a$ which yields
\begin{align*}
&\left(\left(\left(n.a\right)\oplus c\right)\ominus \left(n.a\right)\right)\oplus (n.a)=(n+1).a\\
&\left(\left(n.a\right)\oplus c\right)\vee (n.a)= (n+1).a\\
&b= (n.a)\oplus c = (n+1).a.
\end{align*}
Now, let $b\leq n.a$. If $n=1$, then $b=0=0.a$ or $b=a=1.a$. Let $n\geq 2$ and let, for all $k<n$, we have $b\leq k.a$ implies that $b=m.a$ for some $m\leq k$.
We assume that $m.a\neq (m+1).a$ for all $m<n$. Otherwise, the proof follows from the assumption.
Consider the elements $b\wedge a$, $b\wedge 2.a,\ldots, b\wedge (n-1).a$ and $b\wedge n.a$.

(2) If there exist integers $k$ and $m$ with $k<m<n$ such that $b\wedge m.a=k.a$, we add $(n-m).a$ to each side of the equation, so that $(b\oplus (n-m).a)\wedge n.a=k.a\oplus (n-m).a$. Therefore, $b\leq (b\oplus (n-m).a)\wedge n.a=k.a\oplus (n-m).a\leq (n-1).a$ and
by the assumption, $b=t.a$ for some integer $t\leq n$.

(3) From $b\wedge a\leq a$ we get that $b\wedge a=0$ or $b\wedge a=a$. If
$b\wedge a=0$, then $b=b\wedge n.a\leq n.(b\wedge a)=0$. If $b\wedge a=a$, then we have $b=a$ or $a<b$.
If $b=0$ and $b=a$, then we have nothing to prove.
Otherwise, $a< b$. Then $a\leq b\wedge 2.a\leq 2.a$, which implies that $b\wedge 2.a=a$ or $b\wedge 2.a=2.a$.
The condition $b\wedge 2.a=a$ by (2) imply that $b=m.a$ for some $m\leq n$.
If $b\wedge 2.a=2.a$, then
$b\geq 2.a$ and so $b=2.a$ or $b>2.a$. Now, consider $2.a\leq b\wedge 3.a\leq 3.a$. In a similar way, we can show that
$b=m.a$ for some $m\leq n$ or $3.a\leq b\wedge 4.a\leq 4.a$.
By finite calculations, we get that $b=m.a$ for some $m\leq n$ or $(n-1).a\leq b\wedge n.a\leq n.a$. It follows that $(n-1).a=b$ or $b=n.a$.
\end{proof}

\begin{prop}\label{pr:5.3}
Let $a$ be an atom of a wEMV-algebra $(M;\vee,\wedge,\oplus,\ominus,0)$. Let $\ominus$ on $\mathbb N$ denote the truncate difference, i.e. $m\ominus n=(m-n)\vee 0$, $m,n \in \mathbb N$.
If we denote by $M_a:=\{m.a\mid m\ge 0\}$, then $M_a$ is a subalgebra of $M$.

In addition, if there is the least integer $m_0$ such that $m_0.a$ is an idempotent element of $M$, then $M_a=\{0,a,\ldots,m_0.a\}$ is an EMV-algebra that is termwise equivalent to the MV-algebra $(M_a;\oplus,\lambda_{a_0},0,m_0.a)$ that is isomorphic to $\Gamma(\frac{1}{m_0}\mathbb Z,1)$, and $m.a\ominus n.a = (m\ominus n).a$ for each $0\le m,n \le m_0$.

Otherwise, $M_a $ is isomorphic to the wEMV-algebra $(\mathbb Z^+;\vee,\wedge,\oplus,0)$, where $m\oplus n=m+n$, $m,n \in \mathbb Z^+$.
\end{prop}

\begin{proof}
Due to Lemma \ref{le:5.2}, it is clear that $M_a$ is closed under $0,\vee,\wedge, \oplus$. We show that it is closed also under $\ominus$. If $m\le n$, then clearly $m.a\ominus n.a=0=(m\ominus n).a$.  In the rest, we assume that $m>n$.

(1) First, let $m_0$ be the least integer $m$ such that $m.a$ is an idempotent of $M$. Let $0\le m,n\le m_0$. Let $n=m-i$ for some $i=1,\ldots,n$. Then $m.a \ominus (m-i).a=j.a$, where $j=0,\ldots,i$. Assume that $j<i$. Then $(m.a \ominus (m-i).a)\oplus (m-i).a = m.a= j.a\oplus (m-i).a=(m+j-i).a <m.a$ when we apply Lemma \ref{le:5.2}. This gives a contradiction, so that $j=i$ and $m.a\ominus n.a= (m-n).a$. Clearly, $M_a$ corresponds to $\Gamma(\frac{1}{m_0}\mathbb Z,1)$.

(2) Now, let any $m.a$ be no idempotent and let $m>n$. Then $m.a\ominus n.a=i.a$ for some integer $i>0$ and due to Lemma \ref{le:5.2}, every $(k+1).a$ is a cover of $k.a$, $k\ge 0$. Hence, as at the end of (1), we conclude that $m.a\ominus n.a=(m-n).a=(m\ominus n).a$ whenever $m>n$.
This implies that $k.a\oplus l.a=(k+l).a$ and $M_a$ is isomorphic to the wEMV-algebra $(\mathbb Z^+;\vee,\wedge,\oplus,0)$ of the positive cone $\mathbb Z^+$, see Example \ref{ex:1}.
\end{proof}

\begin{lem}\label{lem3.1}
Let $(M;\vee,\wedge,\oplus,\ominus,0)$ be a wEMV-algebra and let $a$ be an arbitrary element of $M$. For each $x,y,z\in [0,a]$ we have:
\begin{itemize}
\item[{\rm (i)}] $a\ominus (a\ominus x)=x$;
\item[{\rm (ii)}] $x\wedge y=a\ominus((a\ominus x)\vee (a\ominus y))$;
\item[{\rm (iii)}] $(x\wedge y)\oplus z=(x\oplus z)\wedge (y\oplus z)$;
\item[{\rm (iv)}] $z\ominus (x\wedge y)=(z\ominus x)\vee (z\ominus y)$.
\end{itemize}
\end{lem}

\begin{proof}
Let $x,y,z\in [0,a]$. Then

(i) It follows from Definition \ref{WEMV}(iv).

(ii) By (i), $x\wedge y=\left(a\ominus \left(a\ominus x\right) \right)\wedge \left(a\ominus \left(a\ominus y\right)\right)=
a\ominus \left(\left(a\ominus x\right)\vee \left(a\ominus y\right)\right)$.

(iii) By Proposition \ref{pr:1}(a), $(x\wedge y)\oplus z\leq x\oplus z, x\oplus y$. Now, let $w\in M$ such that $w\leq x\oplus z, x\oplus y$.
Proposition \ref{pr:1}(h) implies that $w\ominus z\leq x\wedge y$ $(w\ominus z)\oplus z\leq (x\wedge y)\oplus z$ and so
$w\leq (x\wedge y)\oplus z$ (by Definition \ref{WEMV}(iv)). Therefore, $(x\wedge y)\oplus z=(x\oplus z)\wedge (y\oplus z)$.

(iv) By Proposition \ref{pr:1}(b), $z\ominus(x\wedge y)\geq z\ominus x,z\ominus y$. Now, let $u\in M$ such that
$u\geq z\ominus x,z\ominus y$. Then $u\oplus x,u\oplus y\geq z$ which imply that
$u\oplus (x\wedge y)=(u\oplus x)\wedge (u\oplus y)\geq z$ (by (iii)). Now, by Proposition \ref{pr:1}(h),
$u\geq z\ominus (x\wedge y)$. It follows that $z\ominus (x\wedge y)=(z\ominus x)\vee (z\ominus y)$.
\end{proof}

\begin{prop}\label{prop3.2}
Let $(M;\vee,\wedge,\oplus,\ominus,0)$ be a wEMV-algebra. For each $a\in M$, $(M;\oplus_a,\lambda_a,0,a)$ is an MV-algebra, where for each
$x,y\in [0,a]$,
\[x\oplus_a y=(x\oplus y)\wedge a,\quad \& \quad \lambda_a(x)=a\ominus x.\]

Moreover, if we put $x\ominus_a y:= a\ominus ((a\ominus x)\oplus_a y))$, see {\rm (\ref{eq:oplus})}, then $x\ominus_a x = x\ominus y$.
\end{prop}

\begin{proof}
Put $x,y,z\in [0,a]$.
\begin{eqnarray*}
(x\oplus_a y)\oplus_a z&=& \left(\left(\left(x\oplus y\right)\wedge a\right)\oplus z \right)=
\left(\left(\left(x\oplus y\right)\oplus z\right)\wedge \left(a\oplus z\right) \right)\wedge a,
\mbox{ by Lemma \ref{lem3.1}(iii)}\\
&=& \left(\left(x\oplus y\right)\oplus z\right)\wedge a.
\end{eqnarray*}
In a similar way, $x\oplus_a (y\oplus_a z)=(x\oplus (y\oplus z))\wedge a$ and so $\oplus_a$ is associative.
Now, we can easily show that $([0,a];\oplus_a,0)$ is a commutative ordered monoid with the neutral element $0$,
and $([0,a];\vee,\wedge,0,a)$ is a bounded
distributive lattice.

We know that $(x\ominus y)\oplus y=x\vee y\leq a$, so
\begin{equation}\label{Eq3.2.1}
(x\ominus y)\oplus y=((x\ominus y)\oplus y)\wedge a=(x\ominus y)\oplus_a y.
\end{equation}
On the other hand,
\begin{eqnarray}\label{Eq3.2.2}
a\ominus \left(\left(\left(a\ominus x\right)\oplus y\right)\wedge a\right)&=&
\left(a\ominus \left(\left(a\ominus x\right)\oplus y \right)\right)\vee (a\ominus a),\mbox{ by Lemma \ref{lem3.1}(iv),}\nonumber \\
&=& \left(a\ominus \left(\left(a\ominus x\right)\oplus y \right)\right)\vee 0=(a\ominus (a\ominus x))\ominus y,\mbox{ by Definition \ref{WEMV}(ix)}\nonumber \\
&=& x\ominus y, \mbox{ by Lemma \ref{lem3.1}(i)}.
\end{eqnarray}
It follows from (\ref{Eq3.2.2}) that
\begin{eqnarray}\label{Eq3.2.3}
y\oplus_a \left(a\ominus \left(\left(a\ominus x\right)\oplus_a y\right) \right)=y\oplus_a(x\ominus y)=y\oplus(x\ominus y)=x\vee y,\mbox{ by (\ref{Eq3.2.1})}.
\end{eqnarray}
In a similar way, $x\oplus_a \left(a\ominus \left(x\oplus_a \left(a\ominus y\right)\right) \right)=x\vee y$.

Finally, let $x,y \le a$. Check
\begin{align*}
x\ominus_a y&=a\ominus ((a\ominus x)\oplus_a y))= a\ominus ((a\ominus x)\oplus y)\wedge a)\\
&= [a\ominus ((a\ominus x)\oplus y)]\vee (a\ominus a)=(a\ominus (a\ominus x))\ominus y\\
&=x\ominus y.
\end{align*}
\end{proof}

Recall that if $f:M_1\ra M_2$ is a map between wEMV-algebras $(M_1;\vee,\wedge,\oplus,\ominus,0)$ and $(M_2;\vee,\wedge,\oplus,\ominus,0)$, then $f$ is a wEMV-homomorphism if $f$ preserves the operations $\vee$, $\wedge$, $\oplus$, $\ominus$ and $0$. Moreover, a non-empty subset $S$ of $M_1$ is a wEMV-subalgebra of the wEMV-algebra $M_1$ if it is closed under operations $\vee$, $\wedge$, $\oplus$, $\ominus$.

Consider the class $\mathsf{wEMV}$ of all wEMV-algebras.
Clearly, $\mathsf{wEMV}$ is a variety. Due to \cite[Thm 12.5]{BuSa}, this variety is even arithmetical which can be demonstrated by the Pixley term $m(x,y,z):=\left(\left(x\ominus y\right)\oplus
z\right)\wedge \left(\left(\left(z\ominus y\right)\oplus x\right)\wedge
\left(x\vee z\right)\right)$.
By Proposition \ref{prop3.2}, we can easily show that $\mathsf{wEMV}$ contains $\mathsf{EMV}_a$, the
class of all wEMV-algebras which are associated with EMV-algebras (for more details see \cite{Dvz}). There is a natural question.
``Is $\mathsf{EMV}_a$ a proper subclass of $\mathsf{wEMV}$?"
According to the following example the answer to this question is positive.

\begin{exm}\label{exm3.3}
Consider the positive cone  $M:=G^+$ of a non-trivial Abelian $\ell$-group $G$. Define $x\ominus y:=0\vee (x-y)$ and $x\oplus y :=x+y$. According to Example \ref{ex:1}, $(M;\vee,\wedge,+,\ominus,0)$ is a wEMV-algebra.  But, its reduct $(M;\vee,\wedge,+,0)$ is not an EMV-algebra, since
for each $x\in M\setminus\{0\}$, we have $x<x+x$, so that for every $x\in G^+\setminus \{0\}$, there is no idempotent $a\in M$ such that $x\le a$, see (E4).
Therefore, $\mathsf{EMV}_a$ is a proper subclass of $\mathsf{wEMV}$.
\end{exm}

A non-empty subset $I$ of a wEMV-algebra $(M;\vee,\wedge,\oplus,\ominus,0)$ is called an {\it ideal} if $I$ is a down set which is closed under $\oplus$. Clearly, by Proposition \ref{pr:1}, we can easily see that $I$ is closed under the operations $\vee$, $\wedge$ and $\ominus$, too. An ideal
$P$ of the wEMV-algebra $M$ is {\it prime} if $x\wedge y\in P$ implies that $x\in P$ or $y\in P$. The set of all prime ideals of $M$ is denoted by
$Spec(M)$.

\begin{lem}\label{lem3.4}
In each wEMV-algebra $(M;\vee,\wedge,\oplus,\ominus,0)$ the following inequality holds:
\begin{equation}
\label{eq3.4} x\ominus z\leq (x\ominus y)\oplus (y\ominus z).
\end{equation}
\end{lem}

\begin{proof}
Let $x,y,z\in M$. Put $a\in M$ such that $x\oplus y \oplus z\leq a$.
Then by Proposition \ref{prop3.2}, consider the MV-algebra $([0,a],\oplus_a,\lambda_a,0,a)$. Let $\ominus_a$ be the well-known binary operation
in this MV-algebra, that is $x\ominus_a y=\lambda_a(\lambda_a(x)\oplus_a y)$ for all $x,y\leq a$. Then by Lemma \ref{lem3.1}(iv) and Definition \ref{WEMV}(iv), we have
\begin{eqnarray}
\label{eq3.4.1} x\ominus_a y&=&a\ominus\left(\left(\left(a\ominus x\right)\oplus y\right)\wedge a\right)=\left(a\ominus\left(\left(a\ominus x\right)\oplus y\right)\right) \vee \left(a\ominus a\right)\nonumber\\
&=&\left(\left(a\ominus x\right)\oplus y\right)=\left(a\ominus (a\ominus x)\right)\ominus y=x\ominus y.
\end{eqnarray}
So, the result follows directly since in the MV-algebra $([0,a],\oplus_a,\lambda_a,0,a)$ we have
\begin{equation*}
x\ominus_a z\leq (x\ominus_a y)\oplus_a (y\ominus_a z)\leq (x\ominus_a y)\oplus (y\ominus_a z)=(x\ominus y)\oplus (y\ominus z).
\end{equation*}
\end{proof}

\begin{prop}\label{prop3.5}
Let $I$ be an ideal of a wEMV-algebra $(M;\vee,\wedge,\oplus,\ominus,0)$.
Then the relation $\thi:=\{(x,y)\in M\times M\mid x\ominus y,y\ominus x\in I\}$ is a congruence relation on $M$.
\end{prop}

\begin{proof}
Clearly, $\thi$ is reflexive and symmetric. Transitivity follows from Lemma \ref{lem3.4}.
Let $z\in M$ and $(x,y)\in \thi$. Put $a\in M$ such that
$x\oplus y,x\oplus z,y\oplus z\leq a$.
By Proposition \ref{prop3.2}, $([0,a],\oplus_a,\lambda_a,0,a)$ is an MV-algebra.
Clearly, $I_a:=I\cap [0,a]$ is an ideal of this
MV-algebra.  Let $*\in\{\vee,\wedge,\ominus\}$. Then $x*z,y*z\in [0,a]$ and we can easily seen
that $(x*z)\ominus_a(y*z)\in I_a\s I$,
(since $I_a$ is an ideal of the MV-algebra $[0,a]$). In a similar way, $(x*z)\ominus_a(y*z)\s I$.
From equation (\ref{eq3.4.1}), we have
\begin{eqnarray}\label{eq3.5}
(x*z,y*z)\in \thi,\quad  *\in\{\vee,\wedge,\ominus\}.
\end{eqnarray}
On the other hand, since $x\oplus y\leq a$, then $(x\oplus y)\in I_a$ and $(x\oplus_a z)\ominus_a(y\oplus_a z)\in I_a$. Now,
$x\oplus z,y\oplus z\leq a$ and equation (\ref{eq3.4.1}) imply that $(x\oplus z)\ominus (y\oplus z)\in I_a\s I$.
By the similar way, we can prove that $(y\oplus z)\ominus (x\oplus z)\in I$.
Therefore, $\thi$ is a congruence relation on the wEMV-algebra $(M;\vee,\wedge,\oplus,\ominus,0)$.
\end{proof}

Let $I$ be an  ideal of a wEMV-algebra $(M;\vee,\wedge,\oplus,\ominus,0)$. The set of all congruence classes with respect to $I$ is denoted by $M/I$. Clearly, $M/I$ together with the natural operations forms a wEMV-algebra, see Proposition \ref{prop3.5}. For simplicity, we use
$x/I$ to denote the class $x/\thi$. Therefore, $(M/I;\vee,\wedge,\oplus,\ominus,0/I)$ is a wEMV-algebra which is called the {\it quotient wEMV-algebra} with respect to $I$.

\begin{prop}\label{prop3.6}
Let $P$ be a prime ideal of a wEMV-algebra $(M;\vee,\wedge,\oplus,\ominus,0)$. Then the lattice $(M/P;\vee,\wedge)$ is a chain.
\end{prop}

\begin{proof}
Let $x,y\in M$. There is $a\in M$ such that $x\oplus y\leq a$. Consider the MV-algebra $([0,a];\oplus_a,\lambda_a,0,a)$. Similarly to the proof of
Proposition \ref{prop3.5}, we can show that $P_a:=P\cap [0,a]$ is a prime ideal of the MV-algebra $[0,a]$ and so the quotient MV-algebra
$[0,a]/P_a$ is a chain (with the natural operations). Without loss of generality, we assume that $x/P_a\leq y/P_a$.
Then $(x\ominus_a y)\in P_a\s P$ and so by equation (\ref{eq3.4.1}) $x\ominus y\in P$. Now, we can easily conclude that
$x/P\leq y/P$ on the wEMV-algebra $(M/P;\vee,\wedge,\oplus,\ominus,0/P)$.
\end{proof}

We can easily check that the converse of Proposition \ref{prop3.6} is also true, that is if $E/I$ is a chain, then $I$ is a prime ideal.


A wEMV-algebra $(M;\vee,\wedge,\oplus,\ominus,0)$ with no non-zero idempotent element is called a {\em strict wEMV-algebra}.

\begin{prop}\label{prop3.7}
Let $(M;\vee,\wedge,\oplus,\ominus,0)$ be a linearly ordered wEMV-algebra. Then it is strict or $M$ is termwise  equivalent to an EMV-algebra with
a top element.
\end{prop}

\begin{proof}
Assume that $M$ is not strict. Then there exists an idempotent element $a\in M\setminus \{0\}$. We claim that $M=[0,a]$. Otherwise, put
$x\in M\setminus [0,a]$. Set $b:=(x\oplus a)\oplus (a\oplus x)$. By Proposition \ref{prop3.2}, $([0,b],\oplus_b,\lambda_b,0,b)$ is an MV-algebra containing $a$ in which $a\oplus_b a=(a\oplus a)\wedge b=a\oplus a=a$ and $b\oplus_b b=b$. That is, $0\leq a\leq b$ is a chain of Boolean elements
in the MV-chain $([0,b],\oplus_b,\lambda_b,0,b)$. It follows that $a=b$ and so $x\leq a$ which is a contradiction. Therefore, $M=[0,a]$.
\end{proof}

In each wEMV-algebra $(M;\vee,\wedge,\oplus,\ominus,0)$, we can easily check that, for each ideal $I$  of $M$ and each non-empty subset
$S\s M$, the ideal of $M$ generated by $\{I\cup S\}$ is the set
$\{x\in M\mid x\leq a\oplus x_1\oplus\cdots \oplus x_n,~ \exists\, n\in\mathbb{N},\exists\, a\in I, \exists\, x_1,\ldots x_n\in S \}$.
Now, let $z\in M\setminus I$. Let $T$ be the set of all ideals of $M$ containing $I$ such that $z\in M\setminus J$. By Zorn's lemma,
$T$ has a maximal element, say $P$. Clearly, $z\notin P$. Let $x\wedge y\in P$ for some $x,y\in M$. We claim that $x\in P$ or $y\in P$.
Otherwise, $z\in\langle P\cup \{x\}\rangle$ and $z\in\langle P\cup \{y\}\rangle$. Then there exist $n\in\mathbb{N}$ and $u,v\in P$ such that $z\leq u\oplus n.x$ and $z\leq v\oplus n.y$.
Let $b\in M$ be such that $2n.(u\oplus v)\oplus n^2.(x\oplus y)\leq b$.
Consider the MV-algebra $([0,b];\oplus_b,\lambda_b,0,b)$. Then we have $z\leq (u\oplus n.x)\wedge (v\oplus n.y)\leq
(u\oplus v\oplus n.x)\wedge (u\oplus v\oplus n.y)$. Since the right hand side of the last inequality belongs to $[0,b]$,
we have $z\leq (u\oplus_b v\oplus n\bullet x)\wedge (u\oplus v\oplus n\bullet y)$, where $1\bullet x=x$ and $n\bullet x=x\oplus_b (n-1)\bullet x$
for all integer $n\geq 2$. Since $([0,b];\oplus_b,\lambda_b,0,b)$ is an MV-algebra by \cite[Prop 1.17(i)]{georgescu},
\begin{equation*}
z\leq (u\oplus_b v)\oplus_b (n\bullet x \wedge n\bullet y)\leq 2n\bullet (u\oplus_b v)\oplus_b n^2\bullet(x\wedge y)\leq
2n.(u\oplus v)\oplus n^2.(x\wedge y) \in P,
\end{equation*}
which is a contradiction. So, $P$ is a prime ideal of the wEMV-algebra $M$. Summing up the above arguments, we have the next proposition.

\begin{prop}\label{prop3.8}
Let $(M;\vee,\wedge,\oplus,\ominus,0)$ be a non-zero wEMV-algebra. Then we have:
\begin{itemize}
\item[{\rm (i)}] For each $x\in M\setminus \{0\}$ there exists a prime ideal
$P$ of $M$ such that $x\notin P$.
\item[{\rm (ii)}]  $\bigcap\{P\mid P\in Spec(M)\}=\{0\}$.
\item[{\rm (iii)}] Any ideal $J$ of $M$ can be represented by the intersection of prime ideals contains $J$.
\end{itemize}
\end{prop}

We note that the binary operation $\oplus$ of a wEMV-algebra $(M;\vee,\wedge,\oplus,\ominus,0)$ is {\it cancellative} if, for all $x,y,z \in M$, $x\oplus y= x\oplus z$ implies $y=z$, and $M$ is said to be a {\it cancellative wEAM-algebra}.

\begin{lem}\label{lem3.9}
{\rm (1)} Let $(M;\vee,\wedge,\oplus,\ominus,0)$ be a cancellative wEMV-algebra. Then it is isomorphic to the wEMV-algebra of the positive cone of some $\ell$-group $(G;+,0)$.

{\rm (2)} In addition, every linearly ordered strict wEMV-algebra $(M;\vee,\wedge,\oplus,\ominus,0)$ is a cancellative wEMV-algebra, and $M$ is isomorphic to the wEMV-algebra of the positive cone of a linearly ordered group $(G;+,0)$.
\end{lem}

\begin{proof}
(1) Let $M$ be a cancellative wEMV-algebra. Since according to Proposition \ref{pr:1}(a), $M$ is naturally ordered, i.e. $x\le y$ iff $y=x\oplus z$ for some $z \in M$, due to the Nakada Theorem, \cite[Prop X.1]{Fuc}, there is an $\ell$-group $G$ such that $M$ is isomorphic to $(G^+;\vee,\wedge,\oplus,\ominus,0)$, where $g_1\oplus g_1 = g_1+g_2$, $g_1\ominus g_2=(g_1-g_2)\vee 0$ (see Example \ref{ex:1}).

(2) Let $M$ be a strict and linearly ordered wEMV-algebra. We claim that the operation $\oplus$ in the commutative monoid $(M;\oplus,0)$ is cancellative. Indeed, assume that $x,y,z\in M$ such that
$x\oplus z=y\oplus z$. Since $M$ is a linearly ordered strict wEMV, for each $u\in M$, there is an element $v\in M$ such that $u\lneqq v$.
Let $a\in M$ be such that $2.(x\oplus y\oplus z)\lneqq a$. Consider the MV-algebra $([0,a];\oplus_a,\lambda_a,0,a)$.
There exists an $\ell$-group $(G_a;+,0)$ with a strong unit $u_a$ such that $([0,a];\oplus_a,\lambda_a,0,a)\cong \Gamma(G_a,u_a)$ (see \cite[Sec 2 and 7]{cdm}). We put $\Gamma(G_a,u_a)=[0,a]$ and so $u_a=a$. Then $x\oplus_a z=(x+z)\wedge a$ and $x\oplus_a z=(x\oplus z)\wedge a$.
Since $x\oplus z\lneqq a$, then $x\oplus_az\lneqq a$ which implies that $x+z\lneqq a$ (otherwise, $a\leq (x+z)$, that is $(x+z)\wedge a=a$).
In a similar way, we can show that $y+z\lneqq a$. Hence, $x+z=x\oplus_a z=y\oplus_a z=y+z$ and so $x=y$, and $M$ is a cancellative wEMV-algebra.

According to (1), there is an $\ell$-group $(G;+,0)$ such that $M$ is isomorphic to the wEMV-algebra of the positive cone $G^+$. Since $M$ is linearly ordered,  $(G;+,0)$ is a linearly ordered group.
\end{proof}

\begin{thm}\label{thm3.10}
Each wEMV-algebra is a subalgebra of an associated wEMV-algebra with top element.
\end{thm}

\begin{proof}
If $M=\{0\}$, then the proof is evident. Let $M\neq \{0\}$.

\noindent
Let $S_1:=\{P\in Spec(M)\mid M/P \text{ has a non-zero idempotent element}\}$ and $S_2:=\{P\in Spec(M)\mid M/P \text{ is strict}\}$.
Then $Spec(M)=S_1\cup S_2$. Also, by Proposition \ref{prop3.8}(ii), we can easily prove that
the map $\varphi:M\ra \prod_{P\in Spec(M)}M/P$ sending $x\in M$ to $(x/P)_{P\in Spec(M)}$ is a one-to-one homomorphism.
On the other hand, $\prod_{P\in Spec(M)}M/P\cong (\prod_{P\in S_1}M/P)\times (\prod_{P\in S_2}M/P)$, so we identify these two wEMV-algebras.
By Propositions \ref{prop3.7} and \ref{prop3.6}, for each $P\in S_1$, $M/P$ is an associated wEMV-algebra and so $\prod_{P\in S_1}M/P$ can be viewed as an EMV-algebra, too.
Note that due to \cite[Thm 3.24]{Dvz}, this associated wEMV-algebra has a top element. Now, let $P\in S_2$. If there is $a\in\mathcal I(M)$ such that $a\notin P$, then clearly
$a/P$ is a non-zero idempotent element of $M/P$, which is a contradiction. That is, $\downarrow \mathcal I(M) \s P$ for all $P\in S_2$, so that $\downarrow \mathcal I(M)\subseteq \bigcap\{P\mid P\in S_2\}$.

Suppose that $P\in S_2$. Then $M/P$ is a linearly ordered strict wEMV-algebra. So, by Lemma \ref{lem3.9}, it is the positive cone of an
$\ell$-group $G_P$. Example \ref{ex:1} entails that $M/P$ can be embedded into an associated wEMV-algebra with top element. Hence, $\prod_{P\in S_2}M/P$ can be embedded into an associated wEMV-algebra with top element, too.

Summing up the above arguments, the wEMV-algebra $M$ is a subalgebra of an associated wEMV-algebra with top element.
\end{proof}

The latter theorem allows us to present a similar representation result as the Basic Representation Theorem \ref{2.4} for EMV-algebras. We recall that if a wEMV-algebra possesses a top element, then $1\ominus x$ is said to be a {\it complement} of $x$.

\begin{thm}\label{th:Repres}
Every wEMV-algebra $M$ either has a top element and so it is an associated wEMV-algebra or it can be embedded into an associated wEMV-algebra $N$ with top element as a maximal ideal of $N$. Moreover, every element of $N$ is either the image of $x\in M$ or is a complement of the image of some element $x \in M$.
\end{thm}

\begin{proof}
If $M$ has a top element, the statement is trivial. So suppose that the wEMV-algebra has no top element. Take $S_1$ and $S_2$ as the sets of prime ideals of $M$ defined in the proof of Theorem \ref{thm3.10}. If $S$ is the set of all prime ideals, then $S=S_1\cup S_2$. If $P \in S_1$, then $M/P$ is an associated linearly ordered EMV-algebra with top element. If $P\in S_2$, then $M/P$ is a linearly ordered strict and consequently cancellative wEMV-algebra without top element which corresponds to a wEMV-algebra of a positive cone $G^+_P$. So it can be embedded into $\Gamma(\mathbb Z\lex G_P,(1,0))$. Denote by $N_0=(\prod_{P\in S_1} M/P)\times (\prod_{P\in S_2} \Gamma(\mathbb Z\lex G_P,(1,0)))$ which is an associated wEMV-algebra with a top element $1$, and according to Theorem \ref{thm3.10}, $M$ can be embedded into $N_0$.

Without loss of generality, we can assume that $M\subset N_0$ is a proper wEMV-subalgebra of $N_0$. We denote by $\ominus$ and $\oplus$ also the binary operations of $N_0$. Denote by $M^*=\{1\ominus x\mid x \in M\}$. We assert that $M\cap M^*=\emptyset$. Indeed, if $1\ominus x=y$ for some $x,y\in M$, then $1=(1\ominus x)\oplus x = x\oplus y$ which says $1=x\oplus y \in M$, a contradiction.

First, we define a binary operation $\odot$ on $N_0$ as $x\odot y :=1\ominus ((1\ominus x)\oplus (1\ominus y))$, $x,y \in N_0$.

\vspace{3mm}
\noindent
{\it Claim} {\it If $x,y \in M$, then $x\odot y \in M$ and $x\ominus y = x\odot (1\ominus y)$.}
\vspace{2mm}

Let $x=(x_P)_{P \in S},y=(y_P)_{P \in S}, 1=(1_P)_{P \in S}\in N_0$. Then
$$1\ominus ((1\ominus x)\oplus (1\ominus y)) = (1_P)_{P \in S}\ominus (((1_P)_{P \in S}\ominus (x_P)_{P \in S})\oplus ((1_P)_{P \in S}\ominus (y_P)_{P \in S})).
$$
If $P\in S_1$, then $1_P\ominus ((1_P\ominus x_P)\oplus (1_P\ominus y_P))\in M/P$ since $M/P$ is an associated wEMV-algebra and applying (\ref{eq:odot1}). If $P\in S_2$, then using calculations in $\Gamma(\mathbb Z \lex G_P,(1,0))$, we have also $1_P\ominus ((1_P\ominus x_P)\oplus (1_P\ominus y_P))\in M/P$. Then $1\ominus ((1\ominus x)\oplus (1\ominus y))\in M$.

In addition, $x\odot (1\ominus y)=1\ominus ((1\ominus x)\oplus (1\ominus(1\ominus y))) = 1\ominus ((1\ominus x)\oplus y)= x\ominus y$ (applying Definition \ref{WEMV}(ix)).

Set $N = M\cup M^*$. We show that $N$ is an associated EMV-subalgebra of $N_0$ which satisfies the conditions of our theorem.

Clearly $N$ contains $M$ and $1$. Let $x,y \in N$. We have three cases: (i) $x=x_0,y=y_0 \in M$. Then $x\vee y,x\wedge y, x\oplus y \in N$. Due to Proposition \ref{prop3.2}, we have $x\ominus y = x\ominus_1 y$ and using (\ref{eq:ominus}) and a similar verification as in Claim, we have $x\ominus y \in M \subset N$.

(ii) $x=1\ominus x_0$, $y=1\ominus y_0$ for some $x_0,y_0\in M$. Then $x\vee y = (1\ominus x_0)\vee (1\ominus y_0)= 1\ominus (x_0\wedge y_0)$, $x\wedge y= 1\ominus (x_0\vee y_0)$ and $x\oplus y = (1\ominus x_0)\oplus (1\ominus y_0) = 1\ominus (x_0\odot y_0)\in N$. Finally, by Claim $(1\ominus x_0)\ominus (1\ominus y_0)= (1\ominus x_0)\odot y_0=y_0\odot (1\ominus x_0)= y_0\ominus x_0 \in M\subset N$.

(iii) $x=x_0$ and $y=1\ominus y_0$ for some $x_0,y_0\in M$. We note that $N_0$ can be viewed also as an EMV-algebra with top element, according to \cite[Lem 5.1]{Dvz}, we have $x\odot (1\ominus y)=x\odot (1\ominus (x\wedge y))$.
Then
$$
x\oplus y = x_0\oplus (1\ominus y_0) = 1\ominus (y_0\odot (1\ominus x_0))= 1\ominus (y_0\odot (1\ominus (x_0\wedge y_0)))=1\ominus (y_0 \ominus (x_0\wedge y_0))\in M^* \subset N.
$$

In addition, we have
$$x\wedge y = x_0\wedge (1\ominus y_0)= x_0\odot ((1\ominus x_0)\oplus (1\ominus y_0)) = x_0\odot (1\ominus (x_0\odot y_0))= x_0\ominus (x_0\odot y_0)\in M \subset N.
$$
Using $x\vee y = 1\ominus ((1\ominus x)\wedge (1\ominus y))= 1\ominus((1\ominus x_0)\wedge y_0)$, we have, due to the latter paragraph, $x\vee y \in N_0$. Moreover, $x\ominus y = x_0\ominus (1\ominus y_0)= x_0\odot y_0\in M$ and $y\ominus x = (1\ominus y_0)\ominus x_0 = 1 \ominus (x_0\oplus y_0)\in M^*$, when we have used (ix) of Definition \ref{WEMV}.

Now, we prove that $M$ is a maximal ideal of $N$. Since $M$ is a wEMV-algebra without top element, $M$ is a proper subset of $N$.  To show that $M$ is an ideal, it is sufficient to assume $y\le x \in M$. If $y=(1\ominus y_0)$, then $1= (1\ominus y_0) \oplus y_0\le x_0\oplus y_0\in M$
which is absurd while $1\notin M$. Therefore, $M$ is a proper ideal of $N$. Now, let $y \in N\setminus M$, then $y=1\ominus y_0$ for some $y_0\in M$. Then the ideal $\langle M,y\rangle$ of $N$ generated by $M$ and $1\ominus y_0$ contains $1$, so that $\langle M,1\ominus y_0\rangle = N$ proving $M$ is a maximal ideal of $N$.
\end{proof}

The associated wEMV-algebra $N$ with top element in the latter theorem is said to be a {\it representing} $M$. We note that all representing associated wEMV-algebras of $M$ are mutually isomorphic.

\begin{thm}\label{thm3.11}
The class $\mathsf{wEMV}$ is the least subvariety of the variety $\mathsf{wEMV}$ containing $\mathsf{EMV}_a$. Moreover, $\mathsf{wEMV}=HSP(C)$, where $C$ is the class of all
linearly ordered wEMV-algebras.
\end{thm}

\begin{proof}
By Example \ref{ex:3}, $\mathsf{wEMV}$ contains $\mathsf{EMV}_a$. Let $\mathsf{V}$ be an arbitrary variety of wEMV-algebras containing
$\mathsf{EMV}$. Then by Theorem \ref{thm3.10}, $\mathsf{wEMV}\s \mathsf{V}$. The second part follows from the proof of Theorem \ref{thm3.10}.
\end{proof}

As it was already mentioned, according to \cite{Kom}, the lattice of subvarieties of the variety $\mathsf{MV}$ of MV-algebras is countably infinite. Di Nola and Lettieri presented in \cite{DiLe1} an equational base of any subvariety of the variety $\mathsf{MV}$ which consists of finitely many MV-equations using only $\oplus$ and $\odot$. We know that we can define a binary operation $\odot$ on $M$, see Claim in the proof of Theorem \ref{th:Repres}. Given $x \in M$ and an integer $n\ge 1$, we define
$$
x^1:=x,\quad x^{n+1}:=x\odot x^n,\quad n\ge 1,
$$
and $x^0:=1$ if $M$ has a top element $1$. We note that the subvariety of MV-algebras generated by the MV-algebra $\Gamma(\mathbb Z\lex\mathbb Z,(1,0))$ has an equational base $(2.x)^2=2.x^2$, see \cite[Thm 5.11]{DiLe}, \cite{DiLe1}; it is the subvariety generated by perfect MV-algebras. Denote by $\mathsf{O}$ the trivial subvariety of wEMV-algebras consisting only of the zero element.

\begin{thm}\label{th:can}
Let $\mathsf{Can}$ denote the class of cancellative wEMV-algebras. Then $\mathsf{Can}$ is a subvariety of the variety $\mathsf{wEMV}$, and a wEMV-algebra $M$ belongs to $\mathsf{Can}$ if and only if $M$ satisfies the identity
$$(x\oplus y)\ominus x= y.
$$
Equivalently, $M\in \mathsf{Can}$ if and only if $M$ is isomorphic to the wEMV-algebra of the positive cone $G^+$ of some $\ell$-group $G$.
In addition, if $\mathbb Z^+=(\mathbb Z^+; \vee,\wedge,\oplus,\ominus,0)$ is the wEMV-algebra of the positive cone $\mathbb Z^+$, then $\mathsf{Can}=Var(\mathbb Z^+)$, and $\mathsf{Can}$ is an atom of the lattice of subvarieties of $\mathsf{wEMV}$.

Moreover, if we denote by  $\mathsf{Perf}$ the subvariety of wEMV-algebras satisfying the equation $(2.x)^2=2.x^2$, then $\mathsf{Perf}$ is a cover of the subvariety $\mathsf{Can}$, and the associated wEMV-algebra $(\Gamma(\mathbb Z\lex \mathbb Z,(1,0)); \vee,\wedge,\oplus,\ominus,0)$ with top element and representing the cancellative wEMV-algebra $\mathbb Z^+$ is a generator of the variety $\mathsf{Perf}$.
\end{thm}

\begin{proof}
Let a wEMV-algebra $M$ satisfy the equation $(x\oplus y)\ominus x = y$. We assert that $M$ is cancellative. If $x\oplus y=x\oplus z$, $x,y,z\in M$, then $y= (x\oplus y)\ominus x=(x\oplus z)\ominus x =z$, so that $M$ is a cancellative wEMV-algebra. If $M$ is a cancellative wEMV-algebra, according to Lemma \ref{lem3.9}(1), $M$ is isomorphic to the wEMV-algebra of the positive cone $(G^+;\vee,\wedge,\oplus,\ominus,0)$ of some $\ell$-group $G$. Whence, $\mathsf{Can}$ is a proper non-trivial subvariety of the variety $\mathsf{wEMV}$. It is well known that the group of integers $\mathbb Z$ generates the variety of Abelian $\ell$-groups, see e.g. \cite[Thm 10.B]{Gla}. Using this fact, and the $HSP$-technique, it is possible to show that the wEMV-algebra of the positive cone $\mathbb Z^+$ generates the variety $\mathsf{Can}$.

Clearly, $\mathsf{O} \subsetneq \mathsf{Can}$ is a subvariety of $\mathsf{Perf}$. Let $\mathsf V$ be a subvariety of wEMV-algebras such that $\mathsf O\subsetneq \mathsf V\subseteq \mathsf{Can}$. Then every non-trivial wEMV-algebra of $\mathsf V$ is cancellative. Let $M\in \mathsf{Can}$ be non-trivial and let $f\in M$ be a non-zero element. Then $\{n.f \mid n\ge 0\}$ is a wEMV-subalgebra of $M$ generated by $f$ and it is isomorphic to the wEMV-algebra of the positive cone $\mathbb Z^+$, which implies $\mathbb Z^+\in \mathsf V$ and thus $\mathsf V=\mathsf{Can}$, and $\mathsf{Can}$ is an atom in the lattice of subvarieties of $\mathsf{wEMV}$.

Let $M \in \mathsf{Perf}$. If $M$ possesses a top element, then $(M;\oplus,\lambda_1,0,1)$ is an MV-algebra satisfying the equation $(2.x)^2=2.x^2$. If $M$ has no top element, let $N$ be its representing associated wEMV-algebra. Without loss of generality we can assume that $M\cup M^*=N$. If $x\in M$, then clearly $(2.x)^2=2.x^2$. If $x\in M^*$, then $x=1\ominus x_0$ for some $x_0\in M$ and $(2.x_0)^2=2.x_0^2$ which entails $(2.(1\ominus x_0))^2=2.(1\ominus x^2_0)$, so that $N\in \mathsf{Perf}$.

If $M$ is cancellative and non-trivial, then $M$ is without top element and is isomorphic to the positive cone wEMV-algebra $G^+$. Its representing wEMV-algebra is isomorphic to the associated wEMV-algebra $N=\Gamma(G,u)$ and it satisfies as an MV-algebra the identity of perfect MV-algebras $(2.x)^2=2.x^2$, therefore, $N$ as a wEMV-algebra, satisfies the identity $(2.x)^2=2.x^2$, and $N \in \mathsf{Perf}$. Henceforth, we conclude that $\mathsf{Can}\subsetneq \mathsf{Perf}$. Take the associated wEMV-algebra with top element $M_0=\Gamma(\mathbb Z\lex \mathbb Z,(1,0))$. Then the cancellative wEMV-algebra $\mathbb Z^+$ is a subalgebra of $M_0$ and $\mathbb Z^+ \in Var(M_0)$, where $Var(M_0)$ is the subvariety of $\mathsf{wEMV}$ generated by $M_0$. Whence, $\mathsf{Can}\subseteq Var(M_0)$. The associated wEMV-algebra with top element $\Gamma(\mathbb Z \lex \mathbb Z,(1,0))$ satisfies the identity $(2.x)^2=2.x^2$, so that $Var(M_0)\subseteq \mathsf{Perf}$. Now, let $M\in \mathsf{Perf}$ be an arbitrary wEMV-algebra. Using Theorem \ref{thm3.10}, we know that $M$ is a subdirect product algebra of $N:=(\prod_{P\in S_1}M/P)\times (\prod_{P\in S_2}M/P)$. Let $N_1:= \prod_{P\in S_1}M/P$ and $N_2:=\prod_{P\in S_2}M/P$. Then $N_1$ is an associated wEMV-algebra with top element satisfying $(2.x)^2=2.x^2$. Therefore, $N_1 \in Var(M_0)$ and $N_2$ is a cancellative wEMV-algebra so that $N_2\in \mathsf{Can}\subseteq Var(M_0)$ which yields $N=N_1\times N_2\in Var(M_0)$ and $\mathsf{Perf}\subseteq Var(M_0)$ which proves that the associated wEMV-algebra $\Gamma(\mathbb Z\lex \mathbb Z,(1,0))$ generates the variety $\mathsf{Perf}$.

In what follows, we show that $\mathsf{Perf}$ is a cover of $\mathsf{Can}$. So let $\mathsf V$ be a subvariety of wEMV-algebras such that $\mathsf{Can}\subseteq \mathsf V\subseteq \mathsf{Perf}$ and
let $M \in \mathsf V \setminus \mathsf{Can}$. There are two cases. (1) If $M$ has a top element, then $M$ is an associated wEMV-algebra with top element so that the termwise MV-algebra belongs to the variety generated by perfect MV-algebras, so that the MV-algebra $\Gamma(\mathbb Z\lex \mathbb Z,(1,0))$ belongs to the variety of MV-algebras generated by $M$, consequently, the associated wEMV-algebra $\Gamma(\mathbb Z\lex \mathbb Z,(1,0))$ belongs to the subvariety $Var(M)$ of wEMV-algebras generated by the wEMV-algebra $M$. As it was established in the latter paragraph, $Var(M)=\mathsf{Perf}$. If $M$ has no top element, use again the subdirect embedding of $M$ into $N=N_1\times N_2$ from the latter paragraph. (2) If $S_1=\emptyset$, then $N_2\in \mathsf{Can}$ and $M$ as a subalgebra of $N_2$ also belongs to $\mathsf{Can}$, an absurd. Hence, $S_1$ is non-empty and there is $P \in S_1$ so that $M/P$ is an associated wEMV-algebra with top element and $M/P \in Var(M)$. As in case (1), $\Gamma(\mathbb Z\lex \mathbb Z,(1,0))\in Var(M/P)\subseteq Var(M)\subseteq \mathsf V $ and therefore, $\mathsf V=\mathsf{Perf}$ which proves that $\mathsf{Perf}$ is a cover of $\mathsf{Can}$.
\end{proof}

\begin{thm}\label{th:subvar}
The lattice of subvarieties of the variety $\mathsf{wEMV}$ is countably infinite.
\end{thm}

\begin{proof}
Due to \cite{Kom}, the lattice of subvarieties of the variety $\mathsf{MV}$ of MV-algebras is countably infinite and in \cite{DiLe1}, there is an equational base of any subvariety of the variety $\mathsf{MV}$ which consists of finitely many MV-equations using only $\oplus$ and $\odot$. Hence, let $\mathsf V_{MV}$ be any subvariety of MV-algebras with a finite equational base $\{f_i(x_1,\ldots,x_n)=g_i(y_1,\ldots,y_m) \mid i=1,\ldots,k\}$, where $f_i,g_i$ are finite MV-terms using only $\oplus$ and $\odot$. Denote by $\mathsf W(\mathsf V_{MV})$ the subvariety of wEMV-algebras which satisfies $\{f_i(x_1,\ldots,x_n)=g_i(y_1,\ldots,y_m) \mid i=1,\ldots,k\}$.

Now, let $\mathsf W$ be any non-trivial subvariety of wEMV-algebras. Let $\mathcal V(\mathsf W)$ denote the system of all wEMV-algebras $(M;\vee,\wedge,\oplus,\ominus,0) \in \mathsf W$ with top element, and let $\mathsf V_{MV}(\mathsf W)$  be the subvariety of MV-algebras generated by equivalent MV-algebras $(M;\oplus,\lambda_1,0,1)$ from $\mathcal V(\mathsf W)$. It has a finite equational base using only $\oplus$ and $\odot$. The system of wEMV-algebras satisfying these identities forms a subvariety $\mathsf W(\mathsf V_{MV}(\mathsf W))\subseteq \mathsf W$.

Take an arbitrary wEMV-algebra $M$ from $\mathsf W$. If $M$ has a top element, then $M\in \mathsf W(\mathsf V_{MV}(\mathsf W))$. If $M$ is without top element, we have an embedding of $M$ into the subdirect product $(\prod_{P\in S_1}M/P)\times (\prod_{P \in S_2}M/P)$. If $S_1$ is non-empty, then $N_1=\prod_{P\in S_1}M/P\in \mathsf W(\mathsf V_{MV}(\mathsf W))$. If $S_2$ is non-empty, then $N_2=\prod_{P\in S_2}M/P\in \mathsf{Can}$. Whence, we have three cases. (1) For each non-trivial $M \in \mathsf W$, $S_2$ is empty, then $\mathsf W\subseteq \mathsf W(\mathsf V_{MV}(\mathsf W))\subseteq \mathsf W$. (2) For each non-trivial $M\in \mathsf W$, $S_1$ is empty, then $\mathsf W\subseteq \mathsf{Can}$ and since $\mathsf{Can}$ is an atom in the lattice of subvarieties of $\mathsf{wEMV}$, see Theorem \ref{th:can}, we have $\mathsf W = \mathsf{Can}$. (3) There is a non-trivial wEMV-algebra $M\in \mathsf W$ such that $S_1$ and $S_2$ are both non-empty. Then $\mathsf{Can}\subset \mathsf W \subseteq \mathsf W(\mathsf V_{MV}(\mathsf W)) \vee \mathsf{Can}\subseteq \mathsf W$ which proves $\mathsf W= \mathsf W(\mathsf V_{MV}(\mathsf W))\vee \mathsf{Can}$.

Summarizing, we see that every subvariety $\mathsf W$ of wEMV-algebras either satisfies some finite system of MV-algebras, so it is $\mathsf W(\mathsf V_{MV})$, or it is equal to $\mathsf W(\mathsf V_{MV})\vee \mathsf{Can}$ for some subvariety $\mathsf V_{MV}$ of MV-algebras. Due to Komori, we see that the lattice of wEMV-subvarieties is countably infinite.
\end{proof}

To illustrate the last mentioned three possibilities, case (1) is true e.g. for the subvariety $\mathsf{Idem}$ of wEMV-algebras determined by $x\oplus x = x$, case (2) for $\mathsf{Can}$, and case (3) for $\mathsf{Idem}\vee \mathsf{Can}$. More generally, we have the following characterization.

\begin{rmk}\label{re:variety}
For each integer $n\ge 1$, we define MV-algebras $L_n=\Gamma(\mathbb Z,n)$ and $K_n=\Gamma(\mathbb Z\lex Z,(n,0))$. It is known, see \cite[Thm 8.4.4]{cdm}, that for every proper variety $\mathsf V_{MV}$ of MV-algebras, there are finite sets $I$ and $J$ such that $I\cup J$ is non-empty and $\mathsf V_{MV}$ is generated by $\{L_i,K_j\colon i\in I,j\in J\}$. Then situation (1) at the end of the proof of Theorem \ref{th:subvar} happens only if $\mathsf V_{MV}(\mathsf W)$ is generated only by finitely many $L_i$'s and no $K_j$.
For situation (3), we have two subcases. (3i) $\mathsf V_{MV}(\mathsf W)$ is generated only by finitely many $L_i$'s, then $\mathsf W = \mathsf(\mathsf W_{MV}(\mathsf W))\vee \mathsf{Can}$. (3ii) $\mathsf V_{MV}(\mathsf W)$ contains at least one generator of the form $L_i$ and at least one generator of the form $K_j$. Then $\mathsf W = \mathsf W(\mathsf  V_{MV}(\mathsf W))\vee \mathsf{Can} = \mathsf W(\mathsf V_{MV}(\mathsf W))$ because the cancellative wEMV-algebra $\mathbb Z^+$ is a subalgebra of the associated wEMV-algebra $K_j$.
\end{rmk}

In what follows, we investigate a question when a wEMV-algebra $M$ and its representing associated wEMV-algebra $N$ with top element belong to the same variety and when not.

\begin{cor}\label{co:var}
Let $M$ be a wEMV-algebra without top element, let $N$ be its representing associated wEMV-algebra with top element, let $\mathsf W$ be a proper variety of wEMV-algebras, and $M \in \mathsf W$.

{\rm 1.} If $\mathsf W$ satisfies case {\rm (1)} or case {\rm (3ii)}, then $N$ belongs to $\mathsf W$.

{\rm 2.} If $\mathsf W$ satisfies case {\rm (2)}, then $N\notin \mathsf W$.

{\rm 3.} If $\mathsf W$ satisfies case {\rm (3i)}, then it can happen that $N\notin \mathsf W$.
\end{cor}

\begin{proof}
Let $M \in \mathsf W$.
Applying Theorem \ref{thm3.10}, we know that $M$ is a subdirect product of $N_0:=(\prod_{P\in S_1}M/P)\times (\prod_{P\in S_2}M/P)$ and let $N_1=\prod_{P\in S_1}M/P$ and $N_2= \prod_{P\in S_2}M/P$.

1. Case (1). Then $S_2=\emptyset$ and $M$ is a subdirect product of $\{M/P\colon P \in S_1\}$. Since every $M/P$ has a top element, $N$ is a subalgebra of $N_1$ and thus $N\in \mathsf W$. Case (3ii). Then $S_1$ and $S_2$ are non-empty. The wEMV-algebra $N_1$ has a top element. Every $M/P$ is cancellative for each $P \in S_2$. But $M/P$ can be isomorphically embedded into $\Gamma(\mathbb Z\lex G_P,(1,0))\subseteq \Gamma(\mathbb Z\lex G_P,(j,0)) \in \mathsf W$, so that $N\in \mathsf W$.

2. Case (2). Then $\mathsf W=\mathsf{Can}$, so that $N\notin \mathsf W$.

3. Case (3i). If $S_2=\emptyset$, then $N \subseteq \prod\{M/P\colon P \in S_1\}\in \mathsf W$. If $S_2$ is non-empty, then $N$ is a subalgebra of $N_1\times \prod \{\Gamma(\mathbb Z\lex G_{P},(1,0))\colon P\in S_2\}$. But $N_1$ has a top element and $\prod \{\Gamma(\mathbb Z\lex G_{P},(1,0))\colon P \in S_2\}\notin \mathsf W$. Whence, $N \notin \mathsf W$.
\end{proof}

In the following remark, we describe some interesting categories of wEMV-algebras and their $\ell$-group representations.

\begin{rmk}\label{re:category}
(1) Denote by $\mathsf{wEMV}_1$ the class of wEMV-algebras with top element.
Applying Proposition \ref{prop3.2} and Mundici's representation of MV-algebras by unital $\ell$-group, the category $\mathsf{wEMV}_1$ is categorically equivalent to the category of MV-algebras and also to the category of unital $\ell$-groups.

(2) The category $\mathsf{Can}$ of cancellative wEMV-algebras is categorically equivalent to the category of Abelian $\ell$-groups.

(3) The category of associated wEMV-algebras without top element is categorically equivalent to the category of $\ell$-groups with a fixed special maximal $\ell$-ideal, see \cite[Thm 6.8]{Dvz}.
\end{rmk}

\begin{rmk}\label{rmk3.12}
By \cite[Cor 1.4.7]{cdm}, an MV-equation is satisfied by all MV-algebras \iff it is satisfied by all linearly ordered MV-algebras.
We can simply check that the following identities hold in each linearly ordered MV-algebras:
\begin{eqnarray}
\label{3.12e2} (x\vee y)\ominus z&=&(x\ominus z)\vee (y\ominus z),\\
\label{3.12e1}  x\oplus y&=&(x\vee y)\oplus (x\wedge y).
\end{eqnarray}
So, they hold in each MV-algebra, too. Now, let $(M;\vee,\wedge,\oplus,\ominus,0)$ be a wEMV-algebra and $x,y,z\in M$.
Let $u\geq x\oplus y\oplus z$ be an element of $M$. In the MV-algebra $([0,u];\oplus,\lambda_u,0,u)$, we have
$(x\vee y)\ominus_u z=(x\ominus_u z)\vee (y\ominus_u z)$ which entails that
$(x\vee y)\ominus z=(x\ominus z)\vee (y\ominus z)$ (by equation (\ref{eq3.4.1})).
So, (\ref{3.12e2}) holds in each wEMV-algebra. In a similar way, we can easily show that $(\ref{3.12e1})$ holds in each wEMV-algebra.

In addition, identity (\ref{3.12e1}) implies the following quasi identity
\begin{equation}\label{3.12e3}
x\wedge y = 0 \Rightarrow x\oplus y = x\vee y
\end{equation}
holding in each wEMV-algebra.
\end{rmk}

Now, given an wEMV-algebra $M$, we introduce two important its subalgebras $M_1$ and $M_2$.

\begin{prop}\label{pr:subalg}
Given a wEMV-algebra $(M;\vee,\wedge,\oplus,\ominus,0)$, we define $M_1:=\downarrow \mathcal I(M)$ and $M_2:=\{x\in M\mid x\wedge y=0,~\forall\, y\in \mathcal I(M)\}$. Then $M_1$ is the biggest associated wEVM-subalgebra of $M$ and $M_2$ is a strict wEMV-subalgebra of $M$.

Moreover, if $x_1\vee x_2=y_1\vee y_2$, where $x_1,y_1\in M_1$ and $x_2,y_2\in M_2$, then $x_1=y_1$ and $x_2\vee y_2$. In addition, $x\in M_1$ and $y\in M_2$ imply $x\wedge y = 0$ and $x\oplus y = x\vee y$. Similarly, if $x_1,y_1\in M_1$ and $x_2,y_2\in M_2$, then $x_1\oplus x_2=y_1\oplus y_2$ entails $x_1=y_1$ and $x_2=y_2$.
\end{prop}

\begin{proof}
Consider an arbitrary wEMV-algebra $(M;\vee,\wedge,\oplus,\ominus,0)$.
Clearly, $M_1:=\downarrow \mathcal I(M)$ is closed under the operations
$\vee$, $\wedge$, $\oplus$, $\ominus$ and $0$ which implies that $M_1$ is a subalgebra of $M$. Also, by definition, $(M_1;\vee,\wedge,\oplus,0)$ is an EMV-algebra and $M_1$ is an ideal of $M$, too. In addition, let $M_1'$ be an associated wEMV-algebra that is a subalgebra of $M$. Then clearly, $M_1'\subseteq M_1$.

Now, let $M_2:=\{x\in M\mid x\wedge y=0,~\forall\, y\in \mathcal I(M)\}$.
Then $0\in M_2$ and by Proposition \ref{pr:1}, $M_2$ is closed under $\wedge$ and $\ominus$. Also, distributivity of $(M;\vee,\wedge)$ implies
that $M_2$ is closed under $\vee$. Let $x,y\in M_2$. Put an arbitrary idempotent element $a\in \mathcal I(M)$.
For $b:=x\oplus y\oplus a$, by Proposition \ref{prop3.2},
$([0,b];\oplus_b,\lambda_b,0,b)$ is an MV-algebra and so by the assumption and \cite[Prop 1.17]{georgescu}, we have
\begin{eqnarray*}
(x\oplus y)\wedge a=(x\oplus_b y)\wedge a\leq (x\wedge a)\oplus_b (y\wedge a)= [(x\wedge a)\oplus (y\wedge b)]\wedge b\le (x\wedge a)\oplus (y\wedge a)=0.
\end{eqnarray*}
Thus, $M_2$ is a subalgebra of the wEMV-algebra $M$. Clearly, $M_2$ does not have any non-zero idempotent element, so that $M_2$ is strict.

Let $x\in M$ such that $x=x_1\vee x_2$ and $x=y_1\vee y_2$,
where $x_1,y_1\in M_1$ and $x_2,y_2\in M_2$. Then
\begin{eqnarray*}
x_1=x_1\wedge (y_1\vee y_2)=(x_1\wedge y_1)\vee (x_1\wedge y_2)=x_1\wedge y_1,
\end{eqnarray*}
and so $x_1\leq y_1$. In a similar way, $y_1\leq x_1$ and so $x_1=x_2$. We can easily show that $x_2=y_2$.

Now, if $x\in M_1$ and $y \in M_2$, we conclude that $x\wedge y = 0$ and  (\ref{3.12e3}) entails $x\oplus y = x\vee y$. Consequently $x_1\oplus x_2=y_1\oplus y_2$ implies $x_1\vee x_2=y_1\vee y_2$, so that $x_1=y_1$ and $x_2\vee y_2$.
\end{proof}

The associated wEMV-subalgebra $M_1$ and a strict wEMV-subalgebra $M_2$ of $M$ play an important role in a decomposition of $M$ as a direct product of $M_1$ and $M_2$.

\begin{thm}\label{thm3.13}
Let $(M;\vee,\wedge,\oplus,\ominus,0)$ be a wEMV-algebra such that the ideal of $M$ generated by $\langle M_1\cup M_2\rangle$ is equal to $M$.
Then $M\cong M_1\times M_2$ as wEMV-algebras.
\end{thm}

\begin{proof}
Let $x\in M$. Then by the paragraph just after Proposition \ref{prop3.7}, there are $x_1\in M_1$ and $x_2\in M_2$ such that
$x\leq x_1\oplus x_2$ (note that $M_1$ and $M_2$ are ideals of $M$).
Put $u\in M$ such that $x\oplus x_1\oplus x_2\leq u$. Then by \cite[Prop 1.17(1)]{georgescu}, in the MV-algebra $([0,u];\oplus_u,\lambda_u,0,u)$, we have $x=x\wedge (x_1\oplus x_2)=x\wedge ((x_1\oplus x_2)\wedge u)=x\wedge (x_1\oplus_u x_2)\leq (x_1\wedge x)\oplus_u (x_2\wedge x)=(x_1\wedge x)\oplus (x_2\wedge x)$. So, we can always assume that $x_1,x_2\leq x$. Then
$x\leq x_1\oplus x_2$ and Proposition \ref{pr:1}(h) entails that
$x\ominus x_1\leq x_2$, whence $x_3:=x\ominus x_1\in M_2$.
It follows that $x_1\oplus x_3=x_1\oplus (x\ominus x_1)=x\vee x_1=x$ (since $x_1\leq x$). Define $\varphi:M\ra M_1\times M_2$ by
$\varphi(x)=(x_1,x_2)$, where $x_1\in M_1$, $x_2\in M_2$ and $x=x_1\oplus x_2$.

(i) If $u\in M_1$ and $v\in M_2$, then by equation (\ref{3.12e1}),
$u\oplus v=u\vee v$.

(ii) $\varphi$ is well defined.
Let $x=x_1\oplus x_2=y_1\oplus y_2$ for some $x_1,y_1\in M_1$ and $x_2,y_2\in M_2$.
By (i), $y_1\oplus y_2=y_1\vee y_2$.
Then
$x_1=x_1\wedge (y_1\vee y_2)= (x_1\wedge y_1)\vee (x_1\wedge y_2)$. Since $x_1\in M_1$, then there is $a\in\mathcal I(M)$
such that $x_1\leq a$ and so $x_1\wedge y_2\leq a\wedge y_2=0$. It follows that $x_1= x_1\wedge y_1$ and so $x\leq y_1$. In a similar way,
$y_1\leq x_1$ which implies that $x_1=y_1$. Similarly, we can show that $x_2=y_2$. It follows that $\varphi$ is well defined.

(iii) $\varphi$ preserves $\vee$, $\wedge$, $\oplus$, $\ominus$ and $0$. Clearly, $\varphi$ preserves $\vee$ and $0$. Let $x,y\in M$.
Then there exist $x_1,y_1\in M_1$ and $x_2,y_2\in M_2$ such that $x=x_1\vee x_2$ and $y=y_1\vee y_2$.

$\varphi(x\oplus y)=\varphi(x_1\oplus x_2\oplus y_1\oplus y_2)=\varphi((x_1\oplus y_1)\oplus (x_2\oplus y_2))=(x_1\oplus y_1,x_2\oplus y_2)=
(x_1,x_2)\oplus (y_1,y_2)=\varphi(x)\oplus \varphi(y)$.

$\varphi(x\wedge y)=\varphi((x_1\vee x_2)\wedge (y_1\vee y_2))=\varphi((x_1\wedge y_1)\vee (x_2\wedge y_2))=(x_1\wedge y_1,x_2\wedge y_2)=
\varphi(x)\wedge \varphi(y)$ (note that $x_1\wedge y_2=0=x_2\wedge y_1$).

By the properties of wEMV-algebras and equation (\ref{3.12e2}), we have $(x_1\vee x_2)\ominus(y_1\vee y_2)=((x_1\vee x_2)\ominus y_1)\wedge ((x_1\vee x_2)\ominus y_2)=
((x_1\ominus y_1)\vee (x_2\ominus y_1))\wedge ((x_1\ominus y_2)\vee (x_2\ominus y_2))$. Also,
$x_2\ominus y_1= x_2\ominus (x_2\wedge y_1)=x_2\ominus 0=x_2$. Similarly, since $x_1\wedge y_2=0$, then
$x_1\ominus y_2=x_1$. So, $(x_1\vee x_2)\ominus(y_1\vee y_2)=((x_1\ominus y_1)\vee x_2)\wedge ((x_2\ominus y_2)\vee x_1)$. It follows that
$\varphi(x\ominus y)=\varphi(((x_1\ominus y_1)\vee x_2)\wedge ((x_2\ominus y_2)\vee x_1))=\varphi((x_1\ominus y_1)\vee x_2)\wedge
\varphi((x_2\ominus y_2)\vee x_1)=(x_1\ominus y_1, x_2)\wedge (x_1,x_2\ominus y_2)=(x_1\ominus y_1,x_2\ominus y_2)=\varphi(x\ominus y)$.

(iii) $\varphi$ is an isomorphism.  Clearly, $\varphi$ is one-to-one and onto.

From (i)--(iii) we conclude that $M\cong M_1\times M_2$.
\end{proof}

We note that in the last theorem, $M_1$ is an associated wEMV-algebra and $M_2$ is a strict wEMV-algebra. So, if $M$ satisfies the conditions
of Theorem \ref{thm3.13}, then $M$ is a direct product of an associated wEMV-algebra and a strict wEMV-algebra.
We can easily prove that the converse also holds.

\begin{cor}\label{cor3.15}
Let $(M;\vee,\wedge,\oplus,\ominus,0)$ be a wEMV-algebra. Then $M\cong M_1\times M_2$ if and only if, for each $x\in M$, the set
$\{x\wedge a\mid a\in\mathcal I(M)\}$ has a greatest element.
\end{cor}

\begin{proof}
Let $M\cong M_1\times M_2$. Then
$M=\langle M_1\cup M_2\rangle$. Put $x\in M$.
There exist two elements $x_1\in M_1$ and $x_2\in M_2$ such that $x=x_1\oplus x_2$. Let $a$ be an arbitrary element of $\mathcal I(M)$.
By part (i) in the proof of Theorem \ref{thm3.13}, we have
$x\wedge a=(x\wedge a)\wedge x=(x\wedge a)\wedge (x_1\oplus x_2)=(x\wedge a)\wedge (x_1\vee x_2)=\left(\left(x\wedge a\right)\wedge x_1 \right)\vee \left(\left(x\wedge a\right)\wedge x_2 \right)=\left(x\wedge a\right)\wedge x_1$. It follows that $x_1$ is the greatest element
of the set $\{x\wedge a\mid a\in\mathcal I(M)\}$.

Conversely, by Theorem \ref{thm3.13}, it suffices to show that $M=\langle M_1\cup M_2\rangle$. Put $x\in M$. Let
$x_1:=\max\{x\wedge a\mid a\in\mathcal I(M)\}$. Then $x_1=x\wedge a$ for some $a\in\mathcal I(M)$. Set $x_2:=x\ominus x_1$.
Note that $x_1\wedge x_2=0$. Indeed, we have $x_1\oplus x_2=x_1\vee x=x$. We claim that $x_2\in M_2$.

(1)  Let $u\in\mathcal I(M)$ and  $z:=x\oplus u$.
According to Proposition \ref{pr:1}(g), we have $x\ominus u=\min\{t\in [0,x]\mid t\oplus (x\wedge u)=x\}$.
Also, it is well known that in the MV-algebra $([0,z];\oplus_z,\lambda_z,0,z)$ we have
$x\ominus_z u=\min\{t\leq x\mid t\oplus_z (x\wedge u)=x\}$. Since $(x\ominus_z u)\leq x$, then $(x\ominus_z u)\oplus (x\wedge u)\leq
x\oplus u\leq z$, then $(x\ominus_z u)\oplus (x\wedge u)=(x\ominus_z u)\oplus_z (x\wedge u)$ which simply implies that
$x\ominus_z u=x\ominus_z u$. So, $(x\ominus_z u)\wedge u=\lambda_z(\lambda_z(x)\oplus_z u)\wedge u=0$
(since $\lambda_z(\lambda_z(x)\oplus_z u)\leq \lambda_z(u)$ and $u$ is a Boolean element (i.e. an idempotent) of the mentioned MV-algebra).
For $u:=a$, we get that $x_1\wedge x_2=0$.

(2) Let $b\in\mathcal I(M)$. By the assumption, $x\wedge (a\vee b)=x\wedge a$.
\begin{eqnarray*}
(x\ominus a)\wedge b&=&\left(x\ominus \left(x\wedge a \right) \right)\wedge b=
\left(x\ominus \left(\left(x\wedge (a\vee b) \right) \right) \right)\wedge b=
\left(x\ominus \left(a\vee b\right) \right)\wedge b\\
&\leq & \left(x\ominus \left(a\vee b\right) \right)\wedge (a\vee b).
\end{eqnarray*}
By part (1), $\left(x\ominus \left(a\vee b\right) \right)\wedge (a\vee b)=0$, which implies that $x_2=x\ominus (x\wedge a)=x\ominus a\in M_2$.
\end{proof}

\begin{rmk}\label{product}
Let $(G^+;\vee,\wedge,\oplus,\ominus,0)$ be a wEMV-algebra.

$(1)$ Let $M\cong A\times B$, where $A$ is an associated wEMV-algebra and $B$ is a strict wEMV-algebra.
Then $\mathcal I(M)\cong \{(a,0)\mid a\in \mathcal I(A)\}$. Clearly, $M_1\cong A\times \{0\}$. Also,
$M_2\cong \{0\}\times B$ (since $B$ does not have any non-zero element).

$(2)$ As a direct corollary, we have that if a wEMV-algebra $M$ admits a decomposition $M\cong M_1\times M_2$, where $M_2$ is a strict part of $M$ and $M_2$ is even cancellative, then there is an $\ell$-group $G$ such that $M_2\cong (G^+;\vee,\wedge,\oplus,\ominus,0)$, where the latter wEMV-algebra is defined in Example \ref{ex:1}.
\end{rmk}

\begin{cor}\label{cor3.16}
Consider the assumptions and notations in Theorem \ref{thm3.13}. The wEMV-algebra $M/M_1$ is a strict wEMV-algebra which is isomorphic to $M_2$.
\end{cor}

\begin{proof}
By the note just before Proposition \ref{prop3.7}, $M_1$ is an ideal of the wEMV-algebra $M$ and so by Proposition \ref{prop3.5},
$M/M_1$ is a wEMV-algebra. According to the proof of Theorem \ref{thm3.13}, for each $x\in M$, there are unique elements $x_1\in M_1$
and $x_2\in M_2$ such that $x=x_1\oplus x_2$. Define $f:M/M_1\to M_2$ by $f(x/M_1)=x_2$. We can easily check that $f$ is an isomorphism.
Since $M_2$ is strict, then $M/M_1$ is strict, too.
\end{proof}

We note that another type of a direct decomposition of a wEMV-algebra using a non-zero idempotent will be present in Remark \ref{re:decomp} below; it will use the representation result Theorem \ref{th:Repres}.

Now, we give two examples of wEMV-algebras satisfying the assumptions of Theorem \ref{thm3.13}.

\begin{exm}\label{exm3.17} Suppose that $(M;\vee,\wedge,\oplus,\ominus,0)$ is a wEMV-algebra. \\
(1)  Let $M$ be a chain and take $x\in M$.

(i) If there exists $x\in M_2\setminus\{0\}$, then for each $a\in\mathcal I(M)$ we have $a\wedge x=0$.
Since $M$ is a chain, $x\leq a$ or $a\leq x$. From $x\leq a$, we get that $x=0$ which is absurd. Hence
$a\leq x$ which implies that $a=0$. Therefore, $\mathcal I(M)=\{0\}$. That is $M=M_2$. (ii) Otherwise,
$M_2=\{0\}$. If $x\in M\setminus M_1$, then $x\geq a$ for all $a\in\mathcal I(M)$ (since $M$ is a chain).
If there is $z\in M$ such that $x<z$, then
\begin{equation*}
0<\lambda_z(x)\wedge a=\lambda_z(x)\odot_{z} a \leq \lambda_z(x)\odot_{z} x=0,\quad  a\in \mathcal I(M),
\end{equation*}
where $\odot_z$ is the well-known binary operation of the MV-algebra $(M;\oplus_z,\lambda_z,0,z)$.
It follows that $\lambda_z(x)$ is a non-zero element of $M_2$ which is a contradiction. So, $x$ is the greatest element of $M$, which
means that $x\oplus x=x\in\mathcal I(M)$ and $M=M_1$. By (i) and (ii), we entail that if $M$ is a chain, then $M=M_1$ or $M=M_2$.
That is, $M=\langle M_1\cup M_2\rangle$.

\vspace{2mm}
\noindent
(2) Let $M$ be the product of a family $\{M^i\}_{i\in T}$ of linearly ordered wEMV-algebras.

(i) Clearly, $\mathcal I(M)=\{(x_i)_{i\in T}\mid x_i\in\mathcal I(M^i),~\forall i\in T\}=\prod_{i\in T}\mathcal I(M^i)$.
It follows that $M_1=\prod_{i\in T} M^i_1$.

(ii) $(x_i)_{i\in T}\in M_2$ iff $(x_i)_{i\in T}\wedge (a_i)_{i\in T}=(0)_{i\in T}$ for all $(a_i)_{i\in T}\in \mathcal I(M)$ iff
$x_i\wedge a_i=0$ for all $a_i\in \mathcal I(M^i)$ and all $i\in T$. Hence $N_2=\prod_{i\in T}M^i_2$.

Let $x=(x_i)_{i\in T}\in M$. For each $i\in T$ by Case 1, $M^i=M^i_1$ or $M^i_2$. Let $T_1:=\{i\in T\mid M^i=M^i_1\}$
and $T_2:=\{i\in T\mid M^i=M^i_2\}$. Then $T=T_1\cup T_2$. For each $i\in T$ assume that
\begin{equation*}
y_i= \left\{\begin{array}{lllll}
x_i & i\in T_1\\
0 & \mbox {otherwise,}
\end{array}\right.
\hspace{2cm}
z_i= \left\{\begin{array}{llll}
x_i  & i\in T_2 \\
0 & \mbox {otherwise.}
\end{array}\right.
\end{equation*}
Then by (i) and (ii), $y:=(y_i)_{i\in T}\in M_1$ and $z:=(z_i)_{i\in T}\in M_2$. Also, $x=y\oplus z$, so $M=\langle M_1\cup M_2\rangle$.
\end{exm}

\section{EMV-algebras and Pierce Sheaves}

In the section, we study sheaves of EMV-algebras. If $M$ is a bounded EMV-algebra, $M$ is termwise equivalent to an MV-algebra on $M$, and the Pierce representation of MV-algebras is studied for example in \cite{FiGe} or in \cite[Part 4]{georgescu}. Theory of sheaf spaces of universal algebras is described in \cite{Dav}. In this part we concentrate to a case when an EMV-algebra $M$ does not have a top element.

First we investigate question on direct decomposability of an EMV-algebra.
We show that every idempotent element $a$ of an EMV-algebra determines a decomposition.

\begin{prop}\label{pr:5.1}
Let $(M;\vee,\wedge,\oplus,0)$ be an EMV-algebra and let $a>0$ be a fixed idempotent of $M$ which is not top element. If $M$ is with top element,
$M$ is isomorphic to the direct product of bounded non-trivial EMV-algebras $([0,a];\vee,\wedge,\oplus,0,)$ and $([0,a'];\vee,\wedge,\oplus,0,)$, where $a'=\lambda_1(a)$, i.e.
$$
M \cong [0,a]\times [0,a'].
$$

If $M$ has no top element, let $(N;\vee,\wedge,\oplus,0)$ be its representing EMV-algebra with top element $1$, where $M$ can be embedded onto a maximal ideal of $N$. For simplicity, let $M\subseteq N$. Let $a'=\lambda_1(a)$.
Then $M$ is isomorphic to the direct product of the bounded non-trivial EMV-algebra $([0,a];\vee,\wedge,\oplus,0,)$ with a proper EMV-algebra $M_1= M\cap [0,a']$, i.e.
$$
M \cong [0,a]\times ([0,a']\cap M).
$$
Moreover, the set $[0,a']\cap M$ is a maximal ideal of the EMV-algebra $([0,a'];\vee,\wedge,\oplus,0)$ with top element $a'$ such that every element $x\in [0,a']\cap M$ is either from $M$ or $x= \lambda_{a'}(x_0)$ for a unique element $x_0\in M$.
\end{prop}

\begin{proof}
The mapping $f_a: M \to [0,a]$ defined by $f_a(x)=(x\wedge a)$, $x\in M$, preserves $0,\vee,\wedge,\oplus, \lambda_b$ ($b \in \mathcal I(M)$) and also $1$ if it exists in $M$, that is, $f_a$ is an EMV-homomorphism from $M$ into the bounded EMV-algebra $([0,a];\vee,\wedge,\oplus,0)$.

If $M$ has a top element, then the mapping $\phi: M\to [0,a]\times [0,a']$ defined by $\phi(x)=(x\wedge a, x\wedge a')$, $x \in M$, is an isomorphisms of the EMV-algebras $M$ and $[0,a]\times [0,a']$, i.e.
$ M \cong [0,a]\times [0,a']$.

Now, let $M$ be a proper EMV-algebra, $(N;\vee,\wedge,\oplus,0)$ be its representing EMV-algebra with top element, $M\subseteq N$. Then $M$ is a maximal ideal of $N$.

The mapping
$\varphi:N\ra [0,a]\times[0,a']$ sending $x\in N$ to $(x\wedge a,x\wedge
a')$ is an isomorphism of EMV-algebras  and
$N\cong[0,a]\times[0,a']$. Set $E:=\varphi(M)=\{(x\wedge a,x\wedge a')\mid
x\in M\}$.
Clearly, $E$ and $M$ are isomorphic EMV-algebras. We claim that
$E=[0,a]\times ([0,a']\cap M)$. If $(x,y)\in [0,a]\times ([0,a']\cap M)$,
then clearly $\varphi(x\vee y)=(x,y)$, which implies that
$[0,a]\times ([0,a']\cap M)\s E$. Conversely, for each $x\in M$, we have
$x\wedge a'\le x\in M$, which gives $x\wedge a'\in M$ because $M$ is an ideal of $N$. Whence, $(x\wedge a,x\wedge a')\in [0,a]\times ([0,a']\cap M)$, so
the claim is true.
We note that $([0,a'];\vee,\wedge,\oplus,0)$ is an EMV-algebra, thus $[0,a']\cap M$ is a proper EMV-algebra, too. Indeed, if $y\in [0,a']\cap M$, then $y\in M$ and there is an idempotent $b\in M$ such that $y\le b$. The element $b\wedge a'\le b$, so that $b\wedge a'$ is an idempotent of $[0,a']\cap M$ such that $y\le b\wedge a'$. In addition, if $b$ is an idempotent of $[0,a']\cap M$, so is of $M$. If we take $y\in [0,b]$, then clearly  $\lambda_b(y)$ is the least element $z \in [0,a']\cap M$ such that $z\oplus y = b$.
Finally, $M \cong [0,a]\times ([0,a']\cap M)$.

Clearly, $[0,a']\cap M$ is closed under $\oplus$. Let $x\in [0,a']\cap M$ and $y\in [0,a']$ such that $y\le x$. Since $M$ is an ideal of the EMV-algebra $N$, $y\in M$, so that $y\in [0,a']\cap M$. Now, let $z\in [0,a]\setminus ([0,a']\cap M)$. Then $z\in N\setminus M$ and the ideal of $N$ generated by $M$ and $z$ has to be $N$ because $M$ is a maximal ideal of $N$. Consequently, the ideal of the EMV-algebra $[0,a']$ generated by $[0,a']\cap M$ and $z$ is equal to $[0,a']$ which means that $[0,a']\cap M$ is a maximal ideal of $[0,a']$. Finally, let $x\in [0,a']$. If $x\in M$, then $x\in [0,a']\cap M$. If $x\in [0,a']\setminus ([0,a']\cap M)$, then $x\in N\setminus M$ and there is a unique element $y_0\in M$ such that $x=\lambda_1(y_0)$. Now, we use that there is a unital Abelian $\ell$-group $(G,u)$ such that $N=\Gamma(G,u)$ which means that $x= u-y_0$. Then $x= a+a' -y_0$ and $x=x\wedge a'= (a\wedge a')+(a'\wedge a') -(y_0\wedge a')= a'-x_0 $, where $x_0=y_0\wedge a'\in [0,a']\cap M$. Clearly that $x=\lambda'_{a'}(x_0)$, where $\lambda'_{a'}(x_0)=\min\{t \in [0,a']\cap M\mid t\oplus x_0=a'\}$ and it finishes the proof.
\end{proof}

We note that an EMV-algebra $M$ with top element is directly indecomposable (i.e. it cannot be expressed as a direct product of two non-trivial EMV-algebras) iff $\mathcal I(M)=\{0,1\}$. If $M$ has no top element it is always decomposable as a product of two non-trivial EMV-algebras.

Recall that a bounded distributive lattice $(L;\vee,\wedge,0,1)$ is called a {\em Stone algebra} if, for any $a\in L$, there exists a
Boolean element $b\in L$ such that $\{x\in L\mid x\wedge a=0\}=\downarrow b$.

\begin{rmk}
Let $(M;\vee,\wedge,\oplus,0)$ be an EMV-algebra.
If $M$ is directly indecomposable, then $M$ has a greatest element $1$ and is termwise equivalent to an MV-algebra, $(M;\oplus,\lambda_1,0,1)$.
From \cite[Lem 4.9]{georgescu} it follows that $(M;\vee,\wedge,\oplus,0)$ is directly indecomposable and $(M;\vee,\wedge,0,1)$ is a
Stone algebra \iff $(M;\vee,\wedge,\oplus,0)$ is termwise equivalent to a linearly ordered MV-algebra.
\end{rmk}

\begin{rmk}\label{re:decomp}
Let $(M;\vee,\wedge,\oplus,\ominus,0)$ be a wEMV-algebra and let $N$ be its representing associated wEMV-algebra with a top element $1$, see Theorem \ref{th:Repres}. We can assume that $M \subseteq N$ and $N$ is equivalent to the MV-algebra $(N;\oplus,\lambda_1,0,1)$. Then Proposition \ref{pr:5.1} can be reformulated and proved verbatim in the same way also for wEMV-algebras. We note that we have non-trivial wEMV-algebras without any non-zero idempotent, see Example \ref{ex:1}.
\end{rmk}

A {\it sheaf space} of sets over $X$ or a {\it sheaf} is a triple $T=(E,\pi,X)$, where $E$ and $X$ are topological spaces and $\pi:E \to X$ is a surjective mapping that is a local homeomorphism, i.e. for all $e \in E$, there exist neighborhoods $U$ of $e$ and $V$ of $\pi(e)$ such that $\pi$ on $U$ is a homeomorphism of $U$ onto $V$. For all $x \in X$, the set $\pi^{-1}(\{x\})$ is a {\it fiber} of $x$. If $U$ is an open set of $X$, a {\it local section over} $U$ is a continuous function $g:U\to E$ such that $g(x) \in \pi^{-1}(\{x\})$ for all $x \in U$. If $U=X$, a local section $g$ is called a {\it global section} of the sheaf.

We note that a sheaf $(E,\pi,X)$ is a {\it sheaf of EMV-algebras} if
\begin{itemize}
\item[{\rm (i)}] each fiber $E_x=\pi^{-1}(\{x\})$ is an EMV-algebra,
\item[{\rm (ii)}] if $E\Delta E=\bigcup_{x\in X} (E_x\times E_x)$ with the induced topology from $E\times E$, then all operations $\oplus,\vee,\wedge$ are continuous from $E\Delta E$ to $E$.
\end{itemize}

If $(M;\vee,\wedge,\oplus,0)$ is a bounded EMV-algebra with top element $1$, then $(M;\oplus,',0,1)$, where $x':=\lambda_1(x)$, $x\in M$, is an MV-algebra which is termwise to $(M;\vee,\wedge,\oplus,0)$ and for MV-algebras there are known their Pierce representation by Boolean sheaves whose stalks are directly indecomposable, see \cite[Sect 4]{georgescu} for more information. Inspired by this result, we present a representation of proper EMV-algebras by Pierce sheaves.

Let $(M;\vee,\wedge,\oplus,0)$ be a proper EMV-algebra and $a\in\mathcal
I(M)$. Let $\mathcal P(\mathcal I(M))$ be the set of all prime ideals of $\mathcal
I(M)$. Define $V_a:=\{P\in \mathcal P(\mathcal I(M))\mid a\notin P \}$. Consider the
relation $\sim_a$ on $M$ define by $x\sim_a y$ if and only if
$x\odot a=y\odot a$. Clearly, $\sim_a$ is an equivalence relation on $M$.
There is $b\in\mathcal I(M)$ such that $a,x,y< b$.
Since $([0,b];\oplus,\lambda_b,0,b)$ is an MV-algebra, by \cite[Lem
4.15]{georgescu}, $x\sim_a y$ if and only if $(x\ominus y)\vee (y\ominus x)\leq
\lambda_b(a)$. Note that, by Lemma 5.1, the operation $x\ominus y:=x\odot
\lambda_b(y)$ is correctly defined and $x\odot y=x\odot_b y$.
It follows that $\sim_a$ is a congruence relation on the MV-algebra
$[0,b]$ and so by Proposition 3.13, it is a congruence relation on the
EMV-algebra $M$.

(1) For each $a\in\mathcal I(M)$, set $M_a:=M/\sim_a$, the quotient
EMV-algebra induced by the congruence relation $\sim_a$.
Also, for simplicity, we denote $x/\sim_a$ by $x/a$.

(2) If $V_b\s V_a$, then $a\leq b$. Otherwise, $a\notin [0,b]$ and so by
\cite[Thm 5.12]{Dvz}, there exists prime ideal $P$ such that
$a\notin P$ and $[0,b]\s P$, which is absurd. So, for each couple of
elements $a,b \in \mathcal I(M)$ with $a\le b$, define $\pi_{a,b}: M_b\ra M_a$ by
$\pi_{a,b}(x/b)=x/a$. It is an onto homomorphism of EMV-algebras. We can
easily see that if $a,b,c\in\mathcal I(M)$ such that
$a\leq b\leq c$, then $\pi_{a,b}\circ \pi_{b,c}=\pi_{a,c}$. Moreover,
$\pi_{a,a}:M_a\ra M_a$ is the identity map on $M_a$.

(3) Let $X$ be the set of all prime ideals of $\mathcal I(M)$ endowed  with
the Stone--Zariski topology (= the hull-kernel topology) $\tau$. The sets $\{V_a\mid a\in\mathcal I(M)\}$ form a base of clopen subsets for this topological space. Since $M$ does not have a top element, the Stone--Zariski topology on $X$ gives a Hausdorff topological space that is locally compact but not compact and every $V_a$ is compact and clopen, see \cite[Lem 4.2, Thm 4.10]{Dvz1}.

Then we can extend the assignment
$V_a\mapsto M_a$ and $V_b\s V_a\mapsto \pi_{a,b}:M_b\ra M_a$ to
a Boolean sheaf $T$ of EMV-algebras, called a {\em Pierce sheaf} of $M$.

\begin{rmk}\label{re:6.1}
Consider the above assumptions.

(i) Let $I$ be an ideal of $\mathcal I(M)$.
 Then $\downarrow I:=\{x\in M\mid x\leq a\ \exists\, a\in I\}$ is an ideal of
$M$, we call it a {\it Stonean ideal}. By Theorem 3.16, $M/\downarrow I$ is  an EMV-algebra. We can easily show that
\begin{equation}
x/\downarrow I=y/\downarrow I \Leftrightarrow x\ominus y,y\ominus x\in
\downarrow I\Leftrightarrow x\ominus y,\ y\ominus x\leq a,\ \exists\, a\in I.
\end{equation}
(ii) Set $E_M:=\{x/\downarrow P\mid x\in M,~ P\in X\}$. Define $\pi:E_M\ra X$ by $\pi(x/\downarrow P)=P$. We show that $\pi$ is a well-defined surjective mapping. Let $x/\downarrow P=y/\downarrow Q$ for some $x,y \in M$. If $z\in x/\downarrow P$, then $z/\downarrow P=z/\downarrow Q$, which yields $\downarrow P= 0/\downarrow P=(z\ominus z)/\downarrow P= z/\downarrow P\ominus z\downarrow Q= z/\downarrow Q\ominus z/\downarrow Q=0/\downarrow Q=\downarrow Q$.
Now, we show that
$\downarrow P=\downarrow Q$ implies that $P=Q$. Indeed, if there exists $a\in P\setminus Q$, then from $a\in \downarrow Q$ we get that $a\leq b$ for some
$b\in Q$ and so $a\in Q$ (since $Q$ is an ideal of $\mathcal I(M)$). It follows that $\pi$ is well defined, moreover, $\pi$ is surjective. Suppose that
$$
U(I,x):=\{x/\downarrow P\mid I\nsubseteq P\},$$
where $I$ is an ideal of $\mathcal I(M)$.
Then $U(\{0\},x)=\emptyset$, and in addition, we have:

(1) $z\in P\in U(I,x)\cap U(J,y)$ implies that $z=x/\downarrow P=y/\downarrow Q$ for some $P,Q\in X$. Hence by (ii), $P=Q$, $I\nsubseteq P$,
$J\nsubseteq Q$ and $z=(x\wedge y)/\downarrow P$. There exist $w_1\in I\setminus P$ and $w_2\in J\setminus P$. Clearly, $w_1\wedge w_2\notin P$ and so $I\cap J\nsubseteq P$ which entails that $z= (x\wedge y)/\downarrow P\in U(I\cap J,x\wedge y)$.

(2) For each $x/\downarrow P\in E_M$, choose an idempotent
$a$ such that $a\notin P$. Then clearly, $x/\downarrow P\in U(\langle a\rangle,x)$.

From (i) and (ii) it follows that $\{U(I,x)\mid I\in \mathrm{Ideal} (\mathcal I(M)),~x\in M\}$ is a base for a topology $\tau'$ on $E_M$. Denote this
topological space by $(E_M,\tau')$.
\end{rmk}

We note that $\{V_a\mid a\in\mathcal I(M)\}$ is a base of the topology $\tau$. Clearly, $V_a\cap
V_b=V_{a\cap b}$, since if
$P\in V_a,V_b$, then $a,b\notin P$ and so $a\wedge b\notin P$. Conversely,
if $P\in V_{a\wedge b}$, then $a\wedge b\notin P$ and so
$a\notin P$ and $b\notin P$, so that $a\in P$ or $b\in P$ which implies that
$a\wedge b\in P$. That is $\{V_a\in a\in\mathcal I(M)\}$ is closed
under finite intersections. Moreover, we can easily show that, for the family $\{V_{a_i}\mid a_i\in\mathcal I(M),i\in J\}$, we have
$\bigcup_{i\in J} V_{a_i}=\{P\in X\mid \{a_i\mid i\in J\}\nsubseteq P\}=\{P\in
X\mid \langle\{a_i\mid i\in J\}\rangle \nsubseteq P\}$.
So, there is an ideal $I$ of $\mathcal I(M)$ such that $\bigcup_{i\in J}
V_{a_i}=V_I:=\{P\in X\mid I\nsubseteq P\}$. In addition, every open set in $\tau$ is of the form $V_I=\{P\in X\mid I\nsubseteq P\}$, where $I$ is an ideal of $\mathcal I(M)$ and vice-versa.

We can easily check that $\pi:E_M\ra X$ sending $x/\downarrow P$ to $P$ is a local homeomorphism from the topological space  $(E_M,\tau')$ to $(X,\tau)$. Consequently, we have the following result:

\begin{thm}\label{th:6.2}
Let $(M;\vee,\wedge,\oplus,0)$ be an EMV-algebra. Then
\begin{itemize}
\item[{\rm (i)}] $T=(E_M,\pi,X)$ is a sheaf of sets over $X$.

\item[{\rm (ii)}] For each $P\in X$, $\pi^{-1}(P)=M/\downarrow P$ is a bounded directly indecomposable EMV-algebra.

\item[{\rm (iii)}] $T=(E_M,\pi,X)$ is a sheaf of bounded EMV-algebras over $X$.

\item[{\rm (iv)}]  For each $x\in M$, the map $\widehat{x}:X\ra E_M$, defined by $\widehat{x}(P)=x/\downarrow P$, is a global section of $T$.
\end{itemize}
\end{thm}

\begin{proof}
(i) This part is straightforward to verify.

(ii) Now, we show $M/\downarrow P$ is a bounded directly indecomposable EMV-algebra. The proof is divided into three steps.

(1) Let $x/\downarrow P$ be an idempotent element of $M/\downarrow P$.
Then $(x\oplus x)/\downarrow P=x/\downarrow P$ and so $(x\oplus x)\ominus x\leq p\in P$. Put $a\in \mathcal I(M)$ such that $x\leq a$.
Then $2.x\leq x\oplus p$. Also, $x\oplus p\leq 2.(x\oplus p)=(2.x)\oplus p\leq (x\oplus p)\oplus p=x\oplus p$, so $x\oplus p\in\mathcal I(M)$.
On the other hand, from $\lambda_a(x)/\downarrow P\in \mathcal I(M/\downarrow P)$,
we can show that $\lambda_a(x)\oplus q\in \mathcal I(M)$ for some $q\in P$. Let $a,p,q\leq b\in \mathcal I(M)$.
Consider the MV-algebra $([0,b];\oplus,\lambda_b,0,b)$.
Then
\begin{eqnarray*}
x\wedge \lambda_a(x)\leq x\wedge \lambda_b(x)=\lambda_b(\lambda_b(x\oplus x)\oplus x)=(x\oplus x)\ominus x\leq p\leq c,
\end{eqnarray*}
where  $c:=p\oplus q$ (note that, $c\leq b$).
It follows that $(x\oplus c)\wedge (\lambda_a(x)\oplus c)=(x\wedge \lambda_a(x))\oplus c\leq c\oplus c=c\in P$.
Clearly, $x\oplus c=x\oplus p\oplus q\in \mathcal I(M)$. In a similar way, $\lambda_a(x)\oplus c\in \mathcal I(M)$ and so
$x\oplus c\in P$ or $\lambda_a(x)\oplus c\in P$. Consequently, $x\in \downarrow P$ or $\lambda_a(x)\in \downarrow P$. That is either
$x/\downarrow P=0/\downarrow P$ or $\lambda_a(x)/\downarrow P=0/\downarrow P$ for all idempotent $a\geq x$.

(2) Now, we prove that $M/\downarrow P$ is a bounded EMV-algebra, i.e. it has a top element. First, let $x$ and $y$ be such elements of $M$ that $x/\downarrow P$ and $y/\downarrow P$ are idempotents of $M/\downarrow P$ and $x/\downarrow P\le y/\downarrow P$. From the previous paragraph we know that we can assume without loss of generality that $x$ and $y$ are idempotents of $M$. In addition, since $x/\downarrow P\le y/\downarrow P$ iff $x\le y\oplus p$ for some idempotent $p \in P$. Hence, we can assume that $x$ and $y$ are idempotents such that $x\le y$. Denote by $x_0=y\ominus x$. Then $x_0$ is an idempotent of $M$ such that $x_0/\downarrow P = y/\downarrow P \ominus x/\downarrow P$. Since $\lambda_y(x_0)\in \mathcal I(M)$, from $x_0\wedge \lambda_y(x_0)=0\in P$, we have either $x_0\in P$ or $\lambda_y(x_0)\in P$, so that either $x_0/\downarrow P =0/\downarrow P$ or $\lambda_y(x_0)/\downarrow P =0/\downarrow P$. In the first case we have $x/\downarrow P=y/\downarrow P$. In the second one from $y=x_0\oplus \lambda_y(x_0)$ we have $y/\downarrow P=x_0/\downarrow P= y/\downarrow P\ominus x/\downarrow P$ which yields $x/\downarrow P=0/\downarrow P$.

(3) Now, assume that $M/\downarrow P$ does not have a top element. Therefore, there exists an infinite sequence $\{a_n\}_n$ of elements of $M$ such that every $a_n/\downarrow P$ is an idempotent of $M/\downarrow P$ and $a_n/\downarrow P<a_{n+1}/\downarrow P$. Due to the last paragraph, we can assume that every $a_n$ is an idempotent of $M$, and due to fact that for each $n$, there is an idempotent $p_n$ such that $a_n \le a_{n+1}\oplus p_n$ which allows us to assume that $a_n\le a_{n+1}$ for each $n\ge 1$. By paragraph (2), we see that in each interval $[0/\downarrow P,a_n/\downarrow P]$ there is no idempotent of $M/\downarrow P$ different of $0/\downarrow P$ and $a_n/\downarrow P$ which is a contradiction with $a_n\downarrow P<a_{n+1}/\downarrow P$. Whence, $M\downarrow P$ is a bounded EMV-algebra.

Finally, we have therefore, $M/\downarrow P$ has only two different idempotent elements, so it is bounded and directly indecomposable.

(iii) By (ii), we have that $T=(E_M,\pi,X)$ is a sheaf whose each fiber $\pi^{-1}(P)$ is a bounded indecomposable EMV-algebra. We have to show that $\oplus$ is continuous; the proof of continuity of $\vee$ and $\wedge$ is similar.

Thus, let $x,y \in M$, $P\in X$, and $\widehat x(P), \widehat y(P) \in \pi^{-1}(\{P\})=M/\downarrow P$ be given. Let $V$ be an open neighborhood of $\widehat{x\oplus y}(P)$. Without loss of generality, let $V=U(I,x\oplus y)$ for some neighborhood $V_I$ of $P$. The set $B=\{(\widehat x(P),\widehat y(P))\mid P \in U(I,x\oplus y)\}$ is an open neighborhood of $(\widehat x(P),\widehat y(P))$ in $E_M \Delta E_M$. For the mapping $\beta: (t,s) \mapsto t\oplus s$, we have $\beta^{-1}(V)=B$, so that $\beta$ is continuous.

(iv) Let $x \in M$. Since $\widehat x(P)=x/\downarrow P \in \pi^{-1}(\{P\})$, it is necessary to show that $\widehat x$ is a continuous mapping. Take an arbitrary open set in $E_M$ of the form $U(I,x)=\{x/\downarrow P\mid I \nsubseteq P\}$, where $I$ is any ideal of $\mathcal I(M)$. Then
\begin{align*}
\widehat x^{-1}(U(I,x))&=\{P \in X \mid x/\downarrow P \in U(I,x)\}\\
&=\{P\in X\mid x/\downarrow P= y/\downarrow Q, y\in M, I\nsubseteq P,Q\}\\
&= \{P \in X\mid I \nsubseteq P\}=V_I,
\end{align*}
which is an open set in the hull-kernel topology on $X$. Whence, each $\widehat x$ is a global section.
\end{proof}

\begin{cor}\label{th:6.3}
Let $(M;\vee,\wedge,\oplus,0)$ be an EMV-algebra and $\widehat{M}:=\{\hat{x}\mid x\in M\}$. Consider the following operations on $\widehat{M}$:
$$(\widehat{x}~ \widehat{*}~ \widehat{y})(P)=\widehat{x}(P)* \widehat{y}(P),\quad \forall *\in\{\vee,\wedge,\oplus\}. $$
Then $(\widehat{M};\widehat{\vee},\widehat{\wedge},\widehat{\oplus},\widehat{0})$ is an EMV-algebra.
\end{cor}

\begin{proof}
It is a direct corollary of Theorem \ref{th:6.2}
\end{proof}

We say that an EMV-algebra $M$ is {\it semisimple} if it is a subdirect product of simple EMV-algebras. It is possible to show that $M$ is semisimple iff the intersection of all maximal ideals of $M$ is the set $\{0\}$. In addition, in \cite[Thm 4.11]{Dvz}, we have characterized semisimple EMV-algebras as EMV-algebras of fuzzy sets where all EMV-operations are defined pointwisely.

We say that an EMV-algebra $M$ satisfies the {\it general comparability property} if it holds for every MV-algebra $([0,a]; \oplus,\lambda_a,0,a)$, i.e. if, for any $a \in \mathcal I(M)$ and $x,y \in [0,a]$, there is an idempotent $e \in [0,a]$ such that $x\wedge e\le y$ and $y\wedge \lambda_a(e) \le x$.

In what follows we show that every semisimple proper EMV-algebra with the general comparability property can be embedded into a sheaf of bounded EMV-algebras on the space $X$.

\begin{thm}\label{th:6.4}
Every semisimple EMV-algebra with the general comparability property can be embedded into a sheaf of bounded EMV-algebras on the space $X$.
\end{thm}

\begin{proof}
Due to \cite[Thm 4.4]{Dvz}, the restriction of any maximal ideal $I$ of $M$ to $I\cap \mathcal I(M)$ gives a maximal ideal of $\mathcal I(M)$, so it belongs to $X$. Conversely, according to \cite[Thm 4.9]{Dvz1}, every prime ideal  $P$ of $\mathcal I(M)$ (hence every maximal ideal of $\mathcal I(M)$) can be extended to a maximal ideal $\downarrow P$. Then $\bigcap \{\downarrow P\mid P \in X\}=\Rad(M)=\{0\}$ which implies that $M$ is a subdirect product of the system $\{M/\downarrow P \mid P\in X\}$ of bounded indecomposable EMV-algebras.
\end{proof}

Now, we present the following representation of EMV-algebras as sections of sheaves.

\begin{thm}\label{th:6.5}
Let $M$ be an EMV-algebra and $X$ be a Hausdorff topological space. If for $x\in M$, there is an ideal $I_x$ of $M$ such that $\bigcap_{x\in X} I_x =\{0\}$ and for all $x\in M$, the set $\{x\in X\mid x \in I_x\}$ is open, then $M$ can be embedded into a sheaf of EMV-algebras on the space $X$.
\end{thm}

\begin{proof}
Let $E=\bigcup_{x\in X}\{M/I_x \times \{x\}\}$ and define a mapping $\pi:E\to X$ by $\pi(a/I_x,x)\mapsto x$, $(a/I_x,x)\in E$. It is a well-defined mapping because if $(a/I_x,x)=(b/I_y,y)$, then $x=y$ and $a/I_x=b/I_x$. In addition, $\pi$ is surjective and $\pi^{-1}(\{x\})=M/I_x$ for each $x \in X$.
For all $a\in M$, define a mapping $\widehat a: X \to E$ by $\widehat a(x)=(a/I_x,x)$, $x \in X$.

We assert that the system $\{\widehat a(U) \mid U \text{ open in } X,\ a \in M\}$ is a base of a topology on $E$. Let $a,b \in M$ and $U,V$ be open in $X$. Since $\{x\in X\mid a\in I_x\}=\{x\in X\mid \widehat a(x)=(0/I_x,x)\}$ is open, then $A=\{x \in X \mid \widehat a(x)=\widehat b(x)\}= \{x\in X\mid (\widehat{a\ominus b})(x)=(0/I_x,x)\}\cap \{x \in X\mid (\widehat{b\ominus a})(x)=(0/I_x,x)\}$ is open in $X$. Whence, $B=A\cap U\cap V$ is also open. For all $w \in B$, $\widehat a(w)=\widehat b(w)$ and $\widehat a(w) \in \widehat a(U)\cap \widehat b(V)$. If $t \in \widehat a(U)\cap \widehat b(V)$, then $\widehat a(\pi(t))=t=\widehat b(\pi(t))$ which yields $\widehat a(B)=\widehat b(B)= \widehat a(U)\cap \widehat b(V)$. So this system is a base of a topology on $E$.  Every mapping $\widehat a$ is continuous. Indeed, choose $b\in M$ and $V$ open in $X$. Then we have $\widehat a^{-1}(\widehat b(V))=\{x\in X\mid \widehat a(x)=(a/I_x,x)\in \widehat b(V)\}=\{x\in V\mid (a/Ix,x)=(b/I_x,x)\}=\{x\in V\mid a\ominus b\in I_x\}\cap\{x\in V\mid b\ominus a\in I_x\}$ is open in $X$. In addition, $\widehat a$ is an open mapping and $\pi$ is a local homeomorphism.

The system $T=(E,\pi,X)$ is thus a sheaf. Now, we show that all operations $\oplus,\vee,\wedge $ are continuous. We verify it only for $\oplus$ and for other operations it is similar. Let $x\in X$ and $\widehat a(x),\widehat b(x)\in\pi^{-1}(\{x\})=M/I_x$. Let $V$ be an open neighborhood of $(\widehat{a\oplus b})(x)$. Without loss of generality, let $V=(\widehat{a\oplus b})(U)$ for some open neighborhood $U$ of $x\in X$. The set $C=\{(\widehat a(u),\widehat b(u))\mid u \in U\}$ is an open neighborhood of $(\widehat a(x),\widehat b(x))$ in $E\Delta E$. The mapping $\alpha: (s,t) \mapsto s\oplus t$ from $E\Delta E$ to $E$  has the property $\alpha^{-1}(V)=C$, so that $\alpha$ is continuous and hence, $\oplus $ is continuous.
\end{proof}

\begin{defn}
A distributive lattice $(L;\vee,\wedge)$ with the least element $0$ is called a {\em weak Stone algebra} if, for each $x\in L$, there is
a Boolean element $a\in L$ such that $[0,a]$ is a Stone algebra. For simplicity, an EMV-algebra which is a weak Stone algebra is called
a {\em Stone EMV-algebra}.
\end{defn}

Let $\{M_i\mid i\in\mathbb N\}$ be a class of Stone MV-algebras. By \cite{Dvz}, we know that $M:=\Sigma_{i\in\mathbb N} M_i$ is an EMV-algebra.
Let $(x_i)_{i\in\mathbb N}$ be an arbitrary element of $\Sigma_{i\in\mathbb N} M_i$. There is $n\in\mathbb N$ such that $x_i=0$ for all
$i\geq n+1$. Let $\{x\in M_i\mid x\wedge x_i=0\}=\downarrow b_i$ for all $i\in\{1,2,\ldots,n\}$. Set $u_i=b_i$, for all $1\leq i\leq n$ and
$u_i=0$ for all $i\geq n+1$. Then $u=(u_i)_{i\in\mathbb N}\in\mathcal I(M)$, $x\leq u$, and we can easily check that
$([0,u];\oplus,\lambda_u,0,u)$ is a Stone MV-algebra. Therefore, $M$ is a Stone EMV-algebra.
\begin{thm}\label{WStone pro}
Let $(M;\vee,\wedge,\oplus,0)$ be a Stone EMV-algebra and $P\in\mathcal P(\mathcal I(M))$. Then
\begin{itemize}
\item[{\rm (i)}] $[0,a]\cap\downarrow P$ is a prime ideal of the MV-algebra $(M;\oplus,\lambda_a,0,a)$;
\item[{\rm (ii)}] $Q:=\downarrow P$ is a prime ideal of EMV-algebra $M$;
\item[{\rm (iii)}] $M$ can be embedded into the MV-algebra of global sections of a Hausdorff Boolean sheaf whose stalks are MV-chains.
\end{itemize}
\end{thm}

\begin{proof}
(i) Let $a\in\mathcal I(M)$ and $Q:=\downarrow P$. Clearly, $Q$ is an ideal of $M$ and $[0,a]\cap Q=\downarrow(P\cap [0,a])$.
Since $[0,a]\cap P$ is a prime ideal of $\mathcal I([0,a])$, then by the assumption and \cite[Lem 4.20]{georgescu},
$Q\cap [0,a]=\downarrow ([0,a]\cap P)$ is a prime ideal of the MV-algebra $([0,a];\oplus,\lambda_a,0,a)$.

(ii) Put $x,y\in M$ such that $x\wedge y\in Q$. Then there exists $a\in\mathcal I(M)$ such that $x,y\leq a$.
Consider the MV-algebra $([0,a];\oplus,\lambda_a,0,a)$. By (i), $Q\cap [0,a]$ is a prime ideal of $[0,a]$, so from
$x,y\in [0,a]$ and $x\wedge y\in Q\cap [0,a]$ it follows that $x\in Q\cap [0,a]$ or $y\in Q\cap [0,a]$, which means that $Q$ is a
prime ideal of $M$.

(iii) First, we show that the natural map $M\to \prod_{P\in X}M/\downarrow P$ is one-to-one, where $X=\mathcal P(\mathcal I(M))$.
Let $x\in M$ be such that $x\in\downarrow P$ for all $P\in X$. If $x\in\mathcal I(M)$, then clearly $x=0$ (since $\bigcap_{P\in X} P=\{0\}$).
Otherwise, if $x\notin \mathcal I(M)$, then by the assumption, there is $a\in\mathcal I(M)$ and $b\in\mathcal I([0,a])$ such that
$x\leq a$, $([0,a];\oplus,\lambda_a,0,a)$ is a Stone MV-algebra and $\{y\in [0,a]\mid y\wedge x=0\}=\downarrow b$.
Put $P\in X$. Then there is $e\in P$ such that $x\leq e$. Clearly, $x\leq a\wedge e\in P$. Set $f:=a\wedge e$. Then
$x\wedge \lambda_a(f)=x\odot \lambda_a(f)=0$ which implies that $\lambda_a(f)\leq b$ and so $\lambda_a(b)\leq f\in P$.
It follows that $\lambda_a(b)\in \bigcap_{P\in X} P=\{0\}$. Thus $b=a$ and so $x=x\wedge a=0$ which is a contradiction.
Therefore, $\bigcap_{P\in X}\downarrow P=\{0\}$, which implies that the natural map $M\to \prod_{P\in X}M/\downarrow P$ is one-to-one.
The rest part of the proof is similar to the proof of Theorem \ref{th:6.4}.
\end{proof}

\section{Conclusion}

EMV-algebras are a common generalization of MV-algebras and generalized Boolean algebras so that the existence of a top element is not assumed a priori. Every EMV-algebra either has a top element and then is equivalent to an MV-algebra or a top element fails but it can be embedded into an EMV-algebra with top element as the maximal ideal of the second one. The class of EMV-algebras is not a variety because it is not closed under forming subalgebras. Therefore, we were looking for an appropriate variety of algebras very closed to EMV-algebras containing the class of EMV-algebras as the least variety. We showed that such a class of algebras is forming by new introduced wEMV-algebras which form a variety. If we added to the language of every EMV-algebra a new derived operation $\ominus$, we obtained a wEMV-algebra associated to the original EMV-algebra. One of the basic result is to show that the variety of EMV-algebras is the least subvariety of the variety of wEMV-algebra containing all associated EMV-algebras, see Theorem \ref{thm3.11}. This was possible due to the fact that every wEMV-algebra can be embedded into some associated wEMV-algebra, Theorem \ref{thm3.10}. A representation of a wEMV-algebra $M$ by an associated wEMV-algebra $N$ with top element, where $M$ can be embedded as a maximal ideal of $N$ was presented in Theorem \ref{th:Repres}. We have shown that we have countably many different subvarieties of wEMV-algebras, see Theorem \ref{th:subvar}.
In addition, we studied a situation when a wEMV-algebra $M$ is isomorphic to a direct product of two subalgebras $M_1$ and $M_2$ of $M$, where $M_1$ is a greatest associated wEMV-subalgebra and $M_2$ is a strict wEMV-subalgebra, see Theorem \ref{thm3.13}.

Finally, we studied the Pierce sheaves of EMV-algebras without top element in Section 4, see Theorem \ref{th:6.2}--\ref{th:6.5}.

\end{document}